\documentclass[12pt,letterpaper]{article}
\usepackage[top=1in, bottom=1in, left=1in, right=1in]{geometry}
\usepackage[utf8]{inputenc}
\usepackage[english]{babel}
\usepackage{amsmath} 
\usepackage{empheq}
\usepackage{enumitem}
\usepackage{bm}
\usepackage{bbm}
\usepackage{amsfonts}
\usepackage{float}
\usepackage{wrapfig}
\usepackage{caption}
\usepackage{subcaption}
\usepackage{amssymb,amsthm}

\usepackage{graphicx}
\usepackage{comment}
\usepackage{xcolor}

\newtheorem{thm}{Theorem}[section]

\newtheorem{defn}[thm]{Definition}

\newtheorem{remark}[thm]{Remark}

\newcommand{\eee}{\mathbb{E}}

\newcommand{\mf}{\mathcal{F}}
\newcommand{\pbm}{\mathbb{P}}

\newcommand{\rn}{\mathbb{R}}

\newcommand{\p}{\mathcal{P}}

\newcommand{\bfxi}{\boldsymbol{\xi}}
\newcommand{\bfx}{\mathbf{x}}

\begin{document}
\title{MFGs for Partially Reversible Investment\\
{\em \large In memory of Larry Shepp, our mentor \& friend}}
\author{Haoyang Cao\thanks{Department of Industrial Engineering and Operations Research, University of California, Berkeley. Email: \{hycao, xinguo\}@berkeley.edu}
 \and Xin Guo\footnotemark[1] 
 }
\date{\today}
\maketitle

\begin{abstract}
This paper analyzes a class of infinite-time-horizon stochastic games with singular controls motivated from the partially reversible problem.
It provides an explicit solution for the mean-field game (MFG), and presents sensitivity analysis to compare the solution for the MFG with that for the single-agent control problem. It shows that in the MFG, model parameters not only affect the optimal strategies as in the single-agent case, but also influence the equilibrium price.  It then establishes that the solution to the MFG is 
an $\epsilon$-Nash Equilibrium to the corresponding $N$-player game, with $\epsilon=O\left(\frac{1}{\sqrt N}\right)$.
 
\end{abstract}
\section{Introduction}
\label{introduction}

The seminal work on fuel follower problem and its variants by Bene{\v{s}}, Shepp, and Witsenhausen \cite{Benes1980}  is one of the landmarks in the  
theoretical development  of singular controls (see also \cite{BC67}). The simple
and insightful solution structures have inspired many follow-up works in stochastic controls. See, for instance,
\cite{BS92}, \cite{EK88},  \cite{HW1987},  \cite{Karatzas83}, \cite{SS1989}, and \cite{davis1994problem}.  Such problems  have had a wide range of
applications, including economics and finance \cite{davis1990portfolio}, \cite{Shreve1994}, \cite{Shepp1995}, \cite{jack2008singular}, \cite{Merhi2008} and \cite{Steg2012}, operations research \cite{Guo2011}, and queuing theory \cite{VW1985} and \cite{AB2006}.

Recently, the pioneering works of  \cite{Lasry2007} and \cite{Huang2006} on mean-field games (MFGs) provide an ingenious aggregation approach for analyzing the otherwise notoriously hard $N$-player stochastic games, and have motivated exponentially growing research interests in both theory and applications. See for instance  \cite{Bensoussan2013}, \cite{Carmonaa}, \cite{Carmona2018}, \cite{Gueant2011}, \cite{Lasry2007}, and the references therein. However, a majority of the theoretical developments in MFGs are within the framework of regular controls which are absolutely continuous, with few exceptions, including \cite{Lacker2015} which  formulates a controlled martingale problem to  establish the solution for  a more general class of MFGs with possibly discontinuous controls.

Compared to regular controls, singular controls that are allowed to be discontinuous provide a more general mathematical framework. Though more natural for practical engineering and economics problems, singular controls are more challenging, especially for deriving explicit solutions: studying singular control problems involves analyzing additional (possibly state-dependent) gradient constraints to the underlying Hamilton-Jacobi-Bellman (HJB) equation. Moreover, the Hamiltonian for singular controls of the finite variation type diverges and the standard stochastic maximal principle fails. 

To overcome these technical difficulties for MFGs with singular controls, 
\cite{Fu2017} adopts the notion of relaxed controls and the techniques developed in \cite{Lacker2015} to prove the existence of solutions to MFGs with singular controls in a finite-time horizon, with approximation analysis from MFGs with purely regular controls.  Under the finite-time horizon setting, \cite{Guo2018b} establishes an $\epsilon$-Nash Equilibrium ($\epsilon$-NE) approximation of N-player games with singular controls of finite variation by MFGs with singular controls of bounded velocity.

Still, very little is known on the solution structure of MFGs with singular controls, except for the recent work of  \cite{GX2018}. They study MFGs of fuel follower problem and derive explicit solutions by exploiting symmetric structure in the cost functional. However,  due to  this symmetry, the optimal strategy for the MFG in \cite{GX2018} coincides with that for the single-agent control problem, i.e., the fuel follower problem in \cite{Benes1980}, with no demonstrated game effect. 

Indeed, there are essential  technical difficulties for deriving explicit solutions without certain symmetry structures in MFGs with singular controls. For instance, for a non-stationary MFG, the time-dependent mean information process  leads to a parabolic HJB equation instead of an elliptic type, even in an infinite-time horizon game. This is different from classical control problems with infinite-time horizon. Moreover, the probabilistic approach of forward-backward stochastic differential equations (FBSDEs) does not work easily for the infinite-time horizon case.

\paragraph{Our work.}

In this paper, we analyze a class of infinite-time-horizon MFGs with singular controls, without symmetric cost structures. 
We take the partially reversible investment model in \cite{Guo2005}, formulate its MFG counterpart, provide an analytical solution to the MFG and study the difference between this MFG with its corresponding single-agent control problem, as well as its relation with the associated $N$-player game.

More specifically, the control problem in \cite{Guo2005} is formulated for a class of real option problems originated in the classical work of \cite{Dixit1994}. It is an optimization problem for a company whose revenue is based on the production level of a certain commodity, modeled by a geometric Brownian motion. The company can decrease its production level with a savage value and increase its production level with an investment cost, hence the term ``partially reversible investment''. That is, the dynamics of the production level at time $t$ is given by
\[dx_t = x_t(\delta dt + \gamma dW_t)+d\xi_t,\quad x_{0-}\sim\mu_0,\] 
where $\mu_0\in\mathcal P^2(\rn)$, and the control $\xi_t$ representing the cumulative change in the production level by time $t$ is singular. The problem is to find an optimal investment strategy $\xi_t$ over an appropriate control set in order to maximize its overall expected net profit
\[\eee\left[\int_0^\infty e^{-rt}[\Pi(x_t)dt-\gamma^+ d\xi^+_t -\gamma^- d\xi^-_t]\right].\]
Here, the discount rate $r>0$, $\Pi(\cdot)$ the  revenue function satisfies the usual Inada condition for utility functions, $\gamma^+$ and $\gamma^-$ are the unit costs of increasing and decreasing the production level respectively, subject to some technical conditions for the well-posedness of the  problem.

In the MFG framework, instead of one company, we consider a continuum of infinitely many indistinguishable companies reacting to the market. We assume that the revenue function $f$ is affected by the aggregated production level made by all the companies on the market, i.e., the game interaction among companies is through  the revenue function $f$. 
We analyze explicitly this MFG, and compare  in details this MFG and the single-agent control problem when the revenue function is of the Cobb-Douglas type. In particular, we show that model parameters 
in the MFG impact both the optimal strategies (as in the single-agent case), and the equilibrium price. 
We then formulate the corresponding $N$-player game, and establish that this MFG solution is an approximation to the $N$-player game in the $\epsilon$-NE sense, with $\epsilon=O\left(\frac{1}{\sqrt N}\right)$.

\paragraph{Impact of mean-field interaction via explicit solution.} 
There are earlier works on explicitly solving MFGs with regular controls and on analysis of game effect.  For instance, \cite{Carmona2015} studies the systemic risk characterizing interaction among banks with common noise. It shows that the mean-field interaction creates stability quantified by the systemic risk. \cite{Lacker2018} shows that heterogeneity among players and the common noise have significant impact on the solution structure of MFGs. In particular, without common noise or the heterogeneity, the mean-field interaction would be factored out of the optimization problem of individual players and the equilibrium strategy in MFGs solution would be similar to the single-agent control problem case.

\paragraph{Outline of the paper.} Section \ref{setup} reviews the classical partially reversible investment problem and formulates mathematically the corresponding MFG; Section \ref{numerics} presents a full derivation of an explicit solution to the MFG, provides sensitivity analysis with respect to model parameters, and compares the MFG with the single-agent control problem; Section \ref{sec:approx} connects this MFG with  the associated $N$-player game.

\section{Problem formulation}
\label{setup}
\subsection{Preliminary: partially reversible investment problem}\label{ssec:ctrl}

The basic idea of the partially reversible investment problem goes as follows. A company profits from producing and selling  a commodity.  The revenue function depends on the production level with fluctuations according to, for instance, the market demand. The company has the flexibility to adjust its production level at any time, with the expansion incurring a cost and the contraction bringing a smaller salvage value. The objective of the company is to choose an optimal investment strategy in terms of its production level to maximize the overall expected net profits.

In \cite{Guo2005}, this partially reversible investment problem is formulated  as follows. 
Take a filtered probability space $(\Omega, \mf,\mathbb F=\{\mf_t\}_{t\geq0}, \p)$ supporting a standard Brownian motion $W=\{W_t\}_{t\geq0}$. Assume that $\mathbb F^W$ is the augmented filtration generated by $W$ that satisfies the usual condition. The production level of a company $x_t$ at time $t$ is characterized by a geometric Brownian motion with an initial distribution $\mu_0\in\mathcal P^2(\rn)$ such that
\[dx_t=x_t(\delta dt+\gamma dW_t),\quad x_{0}\sim\mu_0,\]
where $\delta,\gamma>0$ are  drift and volatility coefficients, representing respectively the average and fluctuation in market demand. 
The production level can be adjusted at any time $t$, and possibly in a discontinuous fashion such that 
\begin{equation}
\label{equation:dynamics}
dx_t = x_t(\delta dt + \gamma dW_t) + d\xi_t,\quad x_{0-}\sim\mu_0,\quad \xi_{0-}=0.
\end{equation}
Here, $\xi_t=\xi^+_t-\xi^-_t$, $\xi^\pm_{0-}=0$ with $\xi^{+}_\cdot$ and $\xi^{-}_\cdot$ adapted and nondecreasing c\'adl\'ag processes  representing  the accumulated increased and decreased production level by time $t$ respectively. (Note that when the control is of finite variation, such decomposition of $\xi_\cdot$ by   $\xi^{+}_\cdot$ and $\xi^{-}_\cdot$ is unique).

The objective of the company is to adjust its production level $x_t$ according to a policy $\xi_\cdot = (\xi^+_\cdot,\xi^-_\cdot)$ chosen from an appropriate admissible control set $\mathcal{U}$, in order to maximize its discounted expected total profit over an infinite-time horizon. That is to find
\begin{equation}\label{eq:ctrl}
v(x)=\sup_{(\xi^+_\cdot,\xi^-_\cdot)\in\mathcal{U}}\eee\left[\int_0^\infty e^{-rt}[\Pi(x_t)dt-\gamma^+ d\xi^+_t -\gamma^- d\xi^-_t]\biggl|x_{0-}=x\right], \quad \forall x>0.
\end{equation}
Here the discount rate $r>0$, $\Pi(\cdot)$ the revenue function satisfies the standard Inada condition for utility functions,  $\gamma^+=p>0$ is the unit investment cost to increase  production level, and $-\gamma^-=p(1-\lambda)$ is the unit gain for reducing  production level, with  $\lambda\in(0,1)$  to ensure no-arbitrage.

Finally, the admissible control set $\mathcal U$ is
\begin{equation}\label{eq:ad-ctrl-set-1}
\begin{aligned}
\mathcal{U}=&\left\{(\xi^+_\cdot,\xi^-_\cdot):\xi^+_\cdot,\, \xi^-_\cdot\text{ nondecreasing c\`adl\`ag processes adapted to }\mathbb F^W,\right.\\
&\hspace{50pt}\left.\xi^+_{0-}=\xi^-_{0-}=0,\, \eee\left[\int_0^\infty e^{-rt}d\xi^+_t\right]<\infty, x_t\ge 0. \right\}
\end{aligned}
\end{equation}

In \cite{Guo2005}, the  smooth fit principle in the sense of \cite{Benes1980} is established via regularity analysis for the value function, and the optimal control 
$\xi^*_\cdot=(\xi^{*,+}_\cdot,\xi^{*,-}_\cdot)$
to \eqref{eq:ctrl} is shown to be of bang-bang type. Moreover, the value function is shown to be the  unique classical $\mathcal C^2$ solution to the following Hamilton-Jacobian-Bellman (HJB) equation,
\begin{equation}
0= \min\{r{v} - \Pi(x) -\delta x {v'} -\frac{1}{2}\gamma^2 x^2  {v''}, p-{v'},{v'} -p(1-\lambda) \} ,
\label{hjb-ctrl}
\end{equation}
where $v'$ and $v''$ denote the first and second order derivatives of $v$ respectively. 

When the revenue function $\Pi(x)$ is of the Cobb-Douglas type, i.e., $\Pi(x)=c\rho x^\alpha$ with constants $\rho>0$, $c>0$ and $\alpha\in(0,1)$, then the optimal control is characterized by two thresholds $0<x_b<x_s<\infty$, which are explicitly given by
\begin{equation*}
\begin{cases}
x_b=\left\{\frac{2c\alpha(y_0^n-y_0^\alpha)}{\gamma^2p(1-m)(n-\alpha)\left[y_0^n-(1-\lambda)y_0\right]}\right\}^{\frac{1}{1-\alpha}}\rho^{\frac{1}{1-\alpha}},\\
x_s=\left\{\frac{2c\alpha y_0^{1-\alpha}(y_0^n-y_0^\alpha)}{\gamma^2p(1-m)(n-\alpha)\left[y_0^n-(1-\lambda)y_0\right]}\right\}^{\frac{1}{1-\alpha}}\rho^{\frac{1}{1-\alpha}},
\end{cases}
\end{equation*}
where
\[m=-\left(\frac{\delta}{\gamma^2}-\frac{1}{2}\right)-\sqrt{\left(\frac{\delta}{\gamma^2}-\frac{1}{2}\right)^2+\frac{2r}{\gamma^2}},\,n=-\left(\frac{\delta}{\gamma^2}-\frac{1}{2}\right)+\sqrt{\left(\frac{\delta}{\gamma^2}-\frac{1}{2}\right)^2+\frac{2r}{\gamma^2}},\]
and $y_0>1$ is a root of the following equation
\[1-\lambda = \frac{(n-1)(\alpha-m)y^{m-1}(y^\alpha-y^n)+(1-m)(n-\alpha)y^{n-1}(y^m-y^\alpha)}{(n-1)(\alpha-m)(y^\alpha-y^n)+(1-m)(n-\alpha)(y^m-y^\alpha)}.\]
The value function is then derived by solving the following QVI via the smooth fit principle,
\begin{equation*}
\begin{cases}
p-{v'}=0,&x\leq x<x_b,\\
r{v} - c\rho x^\alpha -\delta x {v'} -\frac{1}{2}\gamma^2 x^2 {v''}=0,& x_b\leq x\leq x_s,\\
{v'} -p(1-\lambda)=0,&x>x_s.
\end{cases}
\end{equation*}
This bang-bang type of control, that is, the optimal control $\xi^*_\cdot$ characterized a pair of threshold $(x_b,x_s)$, suggests that the company should spend the minimum effort to keep its production level within the interval $[x_b,x_s]$.

\subsection{Formulation of MFG}
It is natural to consider the game version of this partially reversible investment problem. We will first consider an MFG in which
there are infinite number of rational and indistinguishable companies, and derive an explicit solution to this MFG. We
will then compare this (much simpler)
MFG with the single-agent problem (in Section \ref{ssec: comparison}), and study its relation with the corresponding $N$-player game (in Section \ref{sec:approx}).

Let $(\Omega, \mf,\mathbb F=\{\mf_t\}_{t\geq0}, \p)$ be a filtered probability space supporting a standard Brownian motion $W=\{W_t\}_{t\geq0}$. Assume that $\mathbb F^W$ is the augmented filtration generated by $W$ that satisfies the usual condition. As in the single-agent control problem, in the MFG a representative company adjusts its production level $x_t$ according to a policy chosen from the admissible control set $\mathcal{U}$ defined in \eqref{eq:ad-ctrl-set-1} to maximize its discounted total profit over an infinite-time horizon,
\begin{equation}\tag{MFG}\label{example2}
\sup_{(\xi^+_\cdot,\xi^-_\cdot)\in\mathcal{U}}\eee\left[\int_0^\infty e^{-rt}[f (x_t,\mu)dt-p d\xi^+_t +p(1-\lambda)d\xi^-_t]\biggl|x_{0-}=x\right], \quad \forall x>0,
\end{equation}
subject to
\begin{equation}\label{eq:ctrl-dyn}
dx_t = x_t(\delta dt + \gamma dW_t) + d\xi^+_t- d\xi^-_t,\quad x_{0-}\sim\mu_0.\end{equation}

Unlike the single agent problem, the revenue function for a representative company in this game \eqref{example2} depends on {\it both} its own production level $x$ and the aggregation of all other companies, denoted by a probability distribution $\mu$. More precisely, $f(x,\mu)$ the revenue function  of a Cobb-Douglas type takes the form of $f(x,\mu)=F(\mu) x^{\alpha}$ for some $\alpha\in (0,1)$, with $\mu$ being the distribution of the production level in the long run, i.e., $\mu=Law(x_\infty)$.
If we consider the inverse demand function, then the price will be given by 
\[\rho = \rho(\mu)=\eee_{X\sim\mu}[\tilde\rho(X)]=\int (a_0 -a_1 y^{1-\alpha}) \mu (dy),\]
and $F(\mu)=c\rho(\mu)$. Effectively one can write 
\[f(x,\mu)=c\rho x^\alpha.\]
Note that in this MFG companies interact through the revenue function $f$. It is also worth noting that, unlike the revenue function for the single-agent control problem in Section \ref{ssec:ctrl} where the unit price $\rho$ is exogenously given and fixed, here in the game \eqref{example2} $\rho$ is endogenously determined.

We will look for a solution to the \eqref{example2} in the following sense.
\begin{defn}\label{soln-smfg}
If there exists a control $\xi^*_\cdot = (\xi^{+,*}_\cdot,\xi^{-,*}_\cdot)\in\mathcal{U}$ and $\rho^*>0$ such that
\begin{enumerate}
\item Under $\rho^*$, $\xi^*_\cdot$ is an optimal control for
\begin{equation}
\tilde{v}(x) = \sup_{(\xi^+_\cdot,\xi^-_\cdot) \in \mathcal{U}} \eee_{\mu_0}\left[\int_0^\infty e^{-rt}[c\rho^*x_t^\alpha dt-p d\xi^+_t + p(1-\lambda)d\xi^-_t]\biggl|x_{0-}=x\right], \, \forall x>0,
\end{equation}
subject to \eqref{eq:ctrl-dyn}.
\item Under $\xi^*_\cdot$ the controlled process $x^*=\{x^*_t\}_{t\geq0}$ given by
\begin{equation}
dx^*_t=x_t^*(\delta dt+\gamma dW_t)+d\xi^{+,*}_t-d\xi^{-,*}_t,\quad x^{*}_{0-}\sim\mu_0
\end{equation}
admits a limiting distribution $\pbm_{x_\infty^*}$ and $\rho^*=\int (a_0 -a_1 y^{1-\alpha}) \pbm_{x_\infty^*} (dy)$.
\end{enumerate}
then the control-mean pair $(\xi^{*}_\cdot,\rho^*)$ is said to be an NE solution to the game \eqref{example2}.

\end{defn}
To ensure the well-posedness of \eqref{example2}, we assume $2\delta+\gamma^2<r$ and $\frac{2\delta}{\gamma^2}\not\in\{\alpha, 1\}$.

\begin{remark}
There is an alternative and equivalent definition of the solution to the game \eqref{example2} . That is, for any fixed $\rho\in\rn$, we may define
\[
\tilde v(\mu_0) = \sup_{(\xi^+,\xi^-)\in\mathcal U}\eee_{\mu_0}\left[\int_0^\infty e^{-rt}\left[c\rho x_t^{\alpha}dt-pd\xi_t^++p(1-\lambda)d\xi^-_t\right]\right]
\]
subject to \eqref{eq:ctrl-dyn}. Then these two solutions are equivalent in the sense that $\tilde v(\mu_0)=\eee_{\mu_0}\left[\tilde v(x_{0-})\right]$.
\end{remark}

\section{Explicit solution to MFG} 
\label{numerics}
\subsection{Solution to the game $\bm{\eqref{example2}}$.}\label{ssec:exp-soln}
We shall now solve the game \eqref{example2}, with the fixed-point approach as in \cite{Lasry2007}.
\paragraph{Step 1. Control problem under fixed mean information.} Fix a $\rho>0$, then the game \eqref{example2} is
a singular control problem, 
\begin{equation}\label{eq:mfg-ctrl}\tag{Control}
\tilde v(x)=\sup_{(\xi^+_\cdot,\xi^-_\cdot)\in\mathcal{U}}\eee\left[\int_0^\infty e^{-rt}[c\rho x^\alpha_tdt-p d\xi^+_t +p(1-\lambda) d\xi^-_t]\biggl|x_{0-}=x\right],\quad x>0,
\end{equation}
subject to \eqref{eq:ctrl-dyn}. The dynamic programming principle leads to the following HJB equation associated with the problem \eqref{eq:mfg-ctrl} under the fixed $\rho$,
\begin{equation}
0= \min\{r\tilde{v} - c x^\alpha \rho -\delta x \partial_x \tilde{v} -\frac{1}{2}\gamma^2 x^2 \partial_{xx} \tilde{v}, p-\partial_x \tilde{v},\partial_x \tilde{v} -p(1-\lambda) \} .
\label{hjbbv}
\end{equation}
Similar to the argument in \cite{Guo2005}, we see that the optimal policy is a bang-bang type and is characterized by an expansion threshold $\tilde x_b$ and a contraction threshold $\tilde x_s$ so that $x_t\in[\tilde x_b,\tilde x_s]$ almost surely. 

More precisely, at time $t=0$, if $x\in(0,\tilde x_b)$, then $\xi^+_0=\tilde x_b-x$ and $\xi^-_0=0$; if $x\in(\tilde x_s,\infty)$, then $\xi^+_0=0$ and $\xi^-_0=x-\tilde x_s$. Note that $x_0=x_{0-}+\xi^+_0-\xi^-_0\in[\tilde x_b,\tilde x_s]$. For $t>0$, it is optimal to impose a minimum amount of adjustment so that $x_t\in[\tilde x_b,\tilde x_s]$. 

Accordingly, the solution $\tilde{v}$ is of the form
\begin{align*}
\tilde{v}(x) = \left\{
\begin{array}{l l} 
 px+C_1, & 0\leq x \leq \tilde{x}_b,\\
 A x^m+Bx^n+Hx^\alpha, & \tilde{x}_b < x <\tilde{x}_s,\\
 p(1-\lambda)x+C_2, & \tilde{x}_s \leq x,
\end{array}\right.
\end{align*}
where $\tilde{x}_b = \inf \{ x : \partial_x \tilde{v}(x) = p\}$, $\tilde{x}_s = \sup \{ x : \partial_x \tilde{v}(x) = p(1-\lambda)\}$ with $0<\tilde x_b\leq\tilde x_s$ (see Lemma 4.4 in \cite{Guo2005}), and since it is assumed that $2\delta+\gamma^2<r$ and hence $\delta<r$,
\[\begin{aligned}
&m=-\left(\frac{\delta}{\gamma^2}-\frac{1}{2}\right)-\sqrt{\left(\frac{\delta}{\gamma^2}-\frac{1}{2}\right)^2+\frac{2r}{\gamma^2}}<0,\,n=-\left(\frac{\delta}{\gamma^2}-\frac{1}{2}\right)+\sqrt{\left(\frac{\delta}{\gamma^2}-\frac{1}{2}\right)^2+\frac{2r}{\gamma^2}}>1,\\
&H=\frac{2c\rho}{\gamma^2(n-\alpha)(\alpha-m)}.
\end{aligned}\]
Moreover, by the smooth-fit principle, we have
\begin{equation}
\begin{cases}
&A\tilde x_b^m+B\tilde x_b^n+H\tilde x_b^\alpha=p\tilde x_b+C_1,\\
&mA\tilde x_b^{m-1}+nB\tilde x_b^{n-1}+\alpha H\tilde x_b^{\alpha-1}=p,\\
&m(m-1)A\tilde x_b^{m-2}+n(n-1)B\tilde x_b^{n-2}+\alpha(\alpha-1) H\tilde x_b^{\alpha-2}=0,\\
&A\tilde x_s^m+B\tilde x_s^n+H\tilde x_s^\alpha=p(1-\lambda)\tilde x_s+C_2,\\
&mA\tilde x_s^{m-1}+nB\tilde x_s^{n-1}+\alpha H\tilde x_s^{\alpha-1}=p(1-\lambda),\\
&m(m-1)A\tilde x_s^{m-2}+n(n-1)B\tilde x_s^{n-2}+\alpha(\alpha-1) H\tilde x_s^{\alpha-2}=0.\end{cases}
\end{equation}
Some algebraic manipulations yield 
\begin{equation}\label{eq:A}
A=\frac{p(n-1)\tilde x_b-\alpha(n-\alpha)H\tilde x_b^\alpha}{m(n-m)\tilde x_b^m}=\frac{p(1-\lambda)(n-1)\tilde x_s-\alpha(n-\alpha)H\tilde x_s^\alpha}{m(n-m)\tilde x_s^m};
\end{equation}
and
\begin{equation}\label{eq:B}
B=\frac{p(m-1)\tilde x_b-\alpha(m-\alpha)H\tilde x_b^\alpha}{n(m-n)\tilde x_b^n}=\frac{p(1-\lambda)(m-1)\tilde x_s-\alpha(m-\alpha)H\tilde x_s^\alpha}{n(m-n)\tilde x_s^n}.
\end{equation}
\begin{figure}[!ht]
\centering
\includegraphics[width=0.75\textwidth]{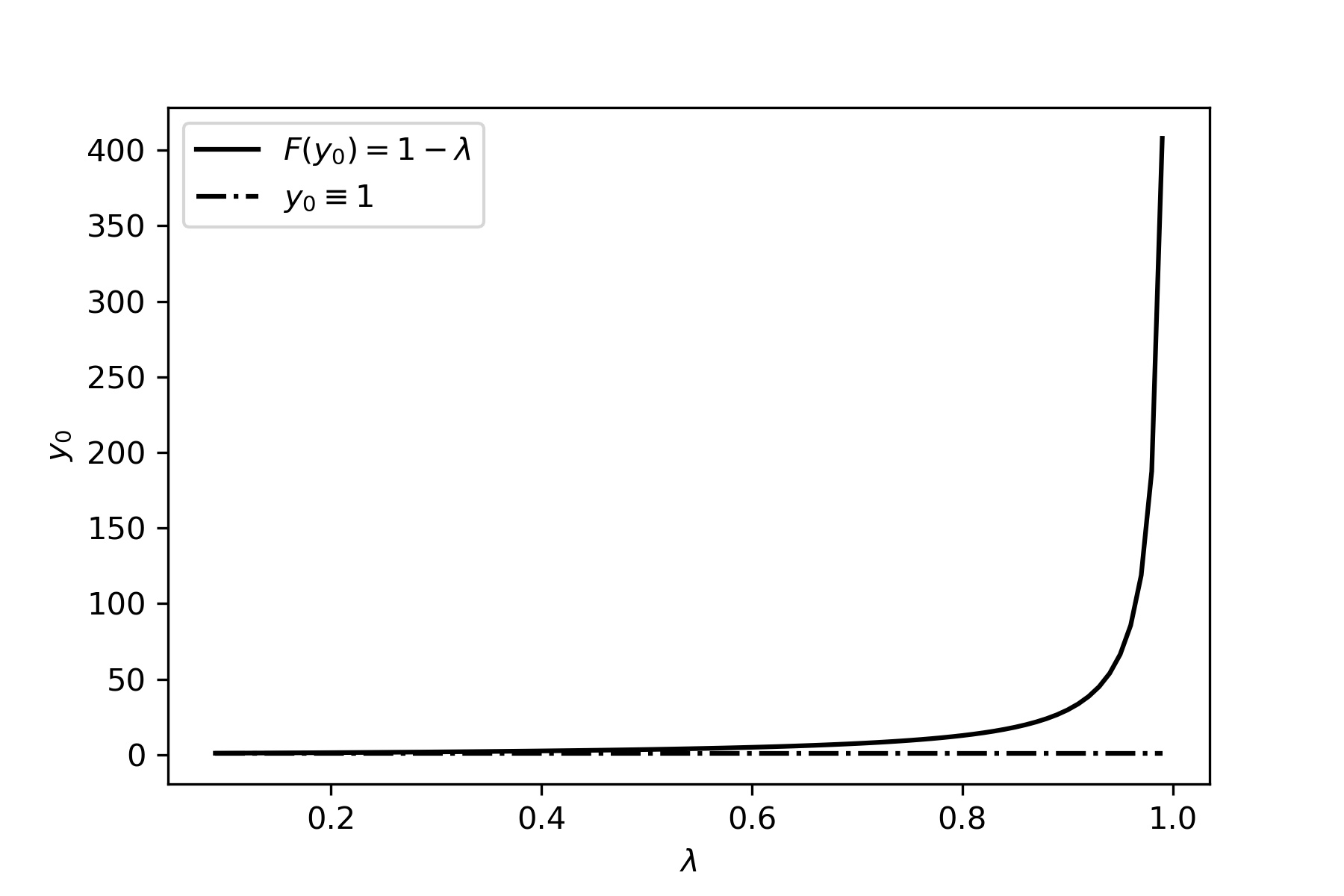}
\caption{$y_0$ increases along with $\lambda$.}
\label{fig:y0-lbd}
\end{figure}
Furthermore, denote $y_0=\frac{\tilde x_s}{\tilde x_b}$ and $y_0\geq1$. By \eqref{eq:A} and $\eqref{eq:B}$, we have 
\begin{empheq}[left=\empheqlbrace]{align}
&p(n-1)\left[(1-\lambda)y_0-y_0^m\right]=\alpha(n-\alpha)H\tilde x_b^{\alpha-1}\left(y_0^\alpha-y_0^m\right),\label{eq:lbd1}\\
&p(m-1)\left[(1-\lambda)y_0-y_0^n\right]=\alpha(m-\alpha)H\tilde x_b^{\alpha-1}\left(y_0^\alpha-y_0^n\right),\label{eq:lbd2}
\end{empheq}
and
\begin{equation}\label{eq: y}
\frac{(n-1)(\alpha-m)y_0^{m-1}(y_0^\alpha-y_0^n)+(1-m)(n-\alpha)y_0^{n-1}(y_0^m-y_0^\alpha)}{(n-1)(\alpha-m)(y_0^\alpha-y_0^n)+(1-m)(n-\alpha)(y_0^m-y_0^\alpha)}=1-\lambda.
\end{equation}
Now, to show that there exists a $y_0$ for \eqref{eq: y}, define $F(y)$ for $y>1$ that
\begin{equation*}F(y)=\frac{(n-1)(\alpha-m)y^{m-1}(y^\alpha-y^n)+(1-m)(n-\alpha)y^{n-1}(y^m-y^\alpha)}{(n-1)(\alpha-m)(y^\alpha-y^n)+(1-m)(n-\alpha)(y^m-y^\alpha)}.\end{equation*}
Since $\lim_{y\to1^+}F(y)=1$, $\lim_{y\to\infty}F(y)=0$, and $F$ is continuous, there exists a $y_0>1$ satisfying $F(y_0)=1-\lambda\in(0,1)$ (see also Figure \ref{fig:y0-lbd}). Note that the function $F$ does not depend on $\rho$, therefore $y_0$ is independent of $\rho$. From \eqref{eq:lbd2}, we can conclude that 
\begin{equation}\label{eq: l}
\tilde x_b=\left\{\frac{2c\alpha(y_0^n-y_0^\alpha)}{\gamma^2p(1-m)(n-\alpha)\left[y_0^n-(1-\lambda)y_0\right]}\right\}^{\frac{1}{1-\alpha}}\rho^{\frac{1}{1-\alpha}}.
\end{equation}
where $\left\{\frac{2c\alpha(y_0^n-y_0^\alpha)}{\gamma^2p(1-m)(n-\alpha)\left[y_0^n-(1-\lambda)y_0\right]}\right\}^{\frac{1}{1-\alpha}}$ does not depend on $\rho$, and 
\begin{equation}\label{eq: u}
\tilde x_s=\tilde x_by_0=\left\{\frac{2c\alpha y_0^{1-\alpha}(y_0^n-y_0^\alpha)}{\gamma^2p(1-m)(n-\alpha)\left[y_0^n-(1-\lambda)y_0\right]}\right\}^{\frac{1}{1-\alpha}}\rho^{\frac{1}{1-\alpha}}.
\end{equation}
After plugging in \eqref{eq: l} and \eqref{eq: u}, $A$ and $B$ are given by \eqref{eq:A} and \eqref{eq:B}, respectively, and 
\begin{equation*}
C_1=A\tilde x_b^m+B\tilde x_b^n+H\tilde x_b^\alpha-p\tilde x_b,\quad C_2=A\tilde x_s^m+B\tilde x_s^n+H\tilde x_s^\alpha-p(1-\lambda)\tilde x_s.\end{equation*}

To justify that the above analytical solution is indeed the solution to the problem \eqref{eq:mfg-ctrl}, one way is via the verification theorem, see for instance \cite{GX2018}. Alternatively, one can first show that the value function is the unique viscosity solution to the corresponding HJB and then establish the uniqueness of a classical $\mathcal C^2$ solution to the HJB, see for instance \cite{Guo2005}. Here we adopt the second approach and claim that under any given $\rho>0$, $\tilde v$ derived above is the value function of problem \eqref{eq:mfg-ctrl}. The proof is similar  to that for  the single-agent case in \cite{Guo2005}, therefore omitted here.

\paragraph{Step 2. Updating the price $\rho$ and the locating the fixed point.} Under any fixed $\rho>0$, the optimal controlled process $x_t$ is a geometric reflected Brownian motion within the interval $[\tilde x_b,\tilde x_s]$. By \cite{Browne1995}, for any $x\in[\tilde x_b,\tilde x_s]$, the scale density is given by
\[s(x)=\exp\left\{-\int_\theta^x\frac{2\delta}{\gamma^2y}dy\right\}=\theta^{\frac{2\delta}{\gamma^2}}x^{-\frac{2\delta}{\gamma^2}},\quad \forall \theta \in(\tilde x_b,\tilde x_s),\]
the speed density is
\[m(x)=\frac{2}{\gamma^2 x^2s(x)}=\frac{2}{\gamma^2\theta^{\frac{2\delta}{\gamma^2}}}x^{\frac{2\delta}{\gamma^2}-2},\]
and finally
\[M(x)=\int_{\tilde x_b}^x m(y)dy=\frac{2}{\gamma^2\theta^{\frac{2\delta}{\gamma^2}}}\frac{x^{\frac{2\delta}{\gamma^2}-1}-\tilde x_b^{\frac{2\delta}{\gamma^2}-1}}{\frac{2\delta}{\gamma^2}-1}.\]
The density function of $P_{x_\infty}$, the limiting distribution of $x_t$, is thus 
\[f(x)=\frac{m(x)}{M(\tilde x_s)}=\frac{\frac{2\delta}{\gamma^2}-1}{\tilde x_s^{\frac{2\delta}{\gamma^2}-1}-\tilde x_b^{\frac{2\delta}{\gamma^2}-1}}x^{\frac{2\delta}{\gamma^2}-2}, \quad\forall x\in[\tilde x_b,\tilde x_s].\]
The updated price $\bar\rho$ under the limiting distribution $\bar\mu=Law(x_\infty)$ is 
\begin{equation}\label{eq:gamma}
\begin{aligned}
\bar\rho=\Gamma(\rho)&=a_0-a_1\int_{\tilde x_b}^{\tilde x_s}x^{1-\alpha}f(x)dx=a_0-a_1\frac{2\delta-\gamma^2}{2\delta-\alpha\gamma^2}\frac{\tilde x_s^{\frac{2\delta}{\gamma^2}-\alpha}-\tilde x_b^{\frac{2\delta}{\gamma^2}-\alpha}}{\tilde x_s^{\frac{2\delta}{\gamma^2}-1}-\tilde x_b^{\frac{2\delta}{\gamma^2}-1}}\\
&=a_0-\rho\cdot a_1\frac{2\delta-\gamma^2}{2\delta-\alpha\gamma^2}\frac{y_0^{\frac{2\delta}{\gamma^2}-\alpha}-1}{y_0^{\frac{2\delta}{\gamma^2}-1}-1}\frac{2c\alpha(y_0^n-y_0^\alpha)}{\gamma^2p(1-m)(n-\alpha)\left[y_0^n-(1-\lambda)y_0\right]},
\end{aligned}
\end{equation}
where the coefficient $a_1\frac{2\delta-\gamma^2}{2\delta-\alpha\gamma^2}\frac{y_0^{\frac{2\delta}{\gamma^2}-\alpha}-1}{y_0^{\frac{2\delta}{\gamma^2}-1}-1}\frac{2c\alpha(y_0^n-y_0^\alpha)}{\gamma^2p(1-m)(n-\alpha)\left[y_0^n-(1-\lambda)y_0\right]}$ does not rely on $\rho$. Clearly, for $a_1$ such that 
\begin{equation}\label{eq:a1}
a_1>0,\ \ a_1\frac{2\delta-\gamma^2}{2\delta-\alpha\gamma^2}\frac{y_0^{\frac{2\delta}{\gamma^2}-\alpha}-1}{y_0^{\frac{2\delta}{\gamma^2}-1}-1}\frac{2c\alpha(y_0^n-y_0^\alpha)}{\gamma^2p(1-m)(n-\alpha)\left[y_0^n-(1-\lambda)y_0\right]}<1,
\end{equation}
the mapping $\Gamma$ is a contraction and therefore admits a unique fixed point 
\begin{equation}\label{eq: rho}
\rho^*=\frac{a_0}{1+a_1\frac{2\delta-\gamma^2}{2\delta-\alpha\gamma^2}\frac{y_0^{\frac{2\delta}{\gamma^2}-\alpha}-1}{y_0^{\frac{2\delta}{\gamma^2}-1}-1}\frac{2c\alpha(y_0^n-y_0^\alpha)}{\gamma^2p(1-m)(n-\alpha)\left[y_0^n-(1-\lambda)y_0\right]}}.
\end{equation}
Substitute $\rho^*$ of \eqref{eq: rho} into \eqref{eq: l} and \eqref{eq: u}, we can derive optimal action boundaries $\tilde x_b^*$ and $\tilde x_s^*$. Denote the singular control characterized by $(\tilde x^*_b, \tilde x^*_s)$ as $\xi^*_\cdot$. Under Definition \ref{soln-smfg}, $(\xi^*_\cdot,\rho^*)$ is a solution to the game \eqref{example2}.

\begin{remark}\label{rmk:l2}
Note that under the assumption $2\delta+\gamma^2<r$, the uncontrolled process $x=\{x_t\}_{t\geq0}$ satisfies $\eee\left[\int_0^\infty e^{-rt}x_t^2dt\right]<\infty$, and this property is preserved for the controlled process $x^*=\{x^*_t\}_{t\geq0}$ under $\xi^*_\cdot$, as it is restricted to a bounded region.
\end{remark}

\subsection{Sensitivity analysis and comparison with single-agent control problem}
\label{ssec: comparison}
As seen from \eqref{eq:gamma}, the iterations  do not stop after the first round, indicating that the game \eqref{example2} demonstrates a genuine game effect from the weak interactions among the players. Moreover, we can see that in the game \eqref{example2}, model parameters $\lambda$, $\delta$, $\gamma$, $r$ and $\alpha$ affect both the optimal strategy of as
in the single-agent control problem \eqref{eq:ctrl} and the equilibrium price $\rho^*$.

To illustrate, consider the following case where $\delta=1$, $\gamma=2$, $r=3$, $\alpha =0.6$, $\lambda = 0.6$, $p=0.5$, $c=1$, $a_0=1$ and $a_1=0.1$. Suppose the iterative process starts from a fixed value $\rho=1$. In the single-agent setting \eqref{eq:ctrl} where the price $\rho=1$ is seen as exogenously given and fixed, the optimal thresholds are given by $x_b=0.053$ and $x_s=0.264$. Figure \ref{fig: bdry-rho} shows that both ${x}_b$ and ${x}_s$ increase along with the value of $\rho$ and the non-action region $[x_b,x_s]$ expands.
\begin{figure}[!ht]
\centering
\includegraphics[width=0.75\textwidth]{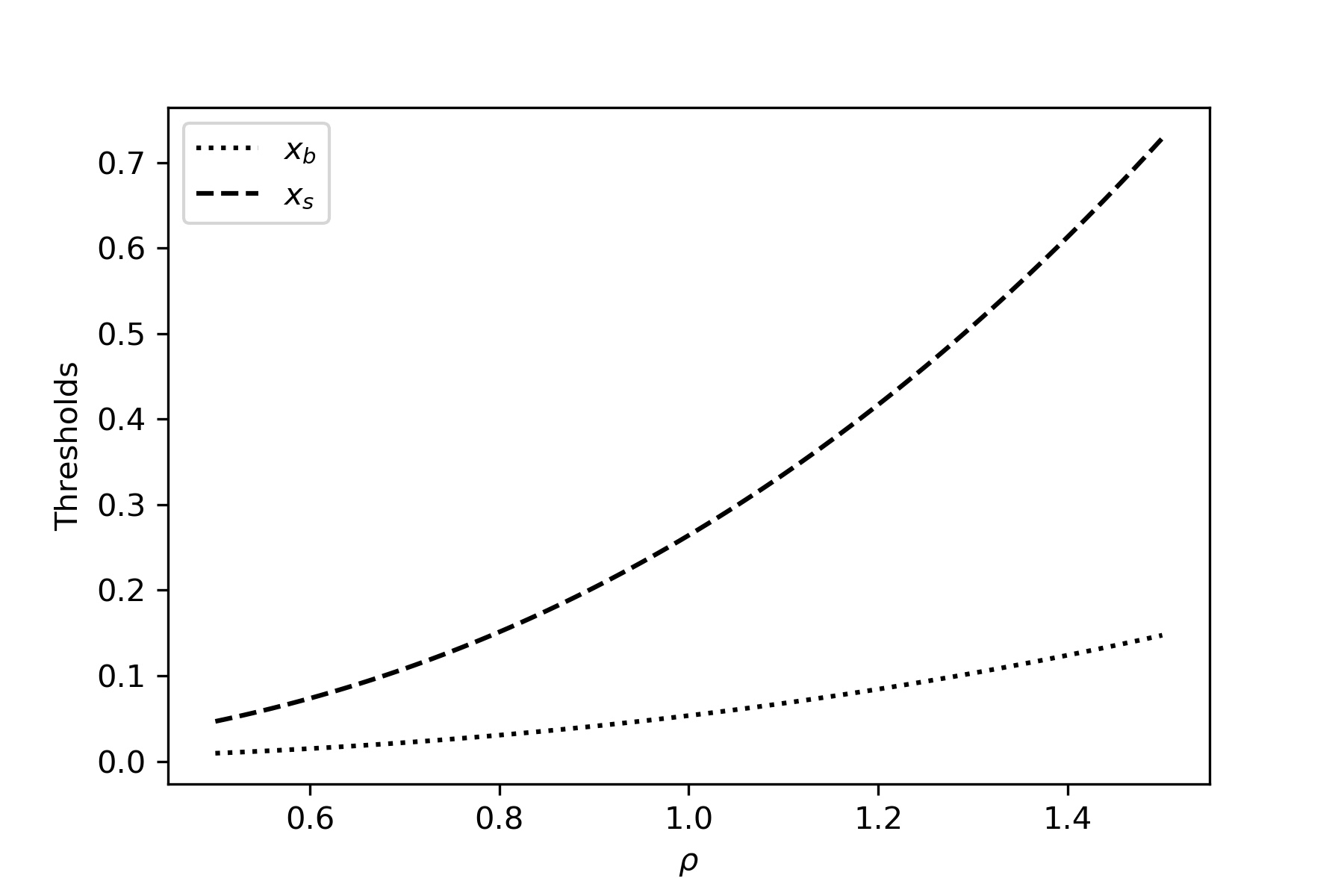}
\caption{Thresholds under different values of $\rho$.}
\label{fig: bdry-rho}
\end{figure}
In the game \eqref{example2}, in contrast, the equilibrium price is $\rho^*=0.96$ under which the optimal thresholds are $\tilde x_b^*=0.048$ and $\tilde x_s^*=0.239$. Figure \ref{fig:1vm} shows the difference in the thresholds of intervention between the single-agent control problem \eqref{eq:ctrl} and the game \eqref{example2}.
\begin{figure}[!ht]
\centering
\includegraphics[width=0.85\textwidth]{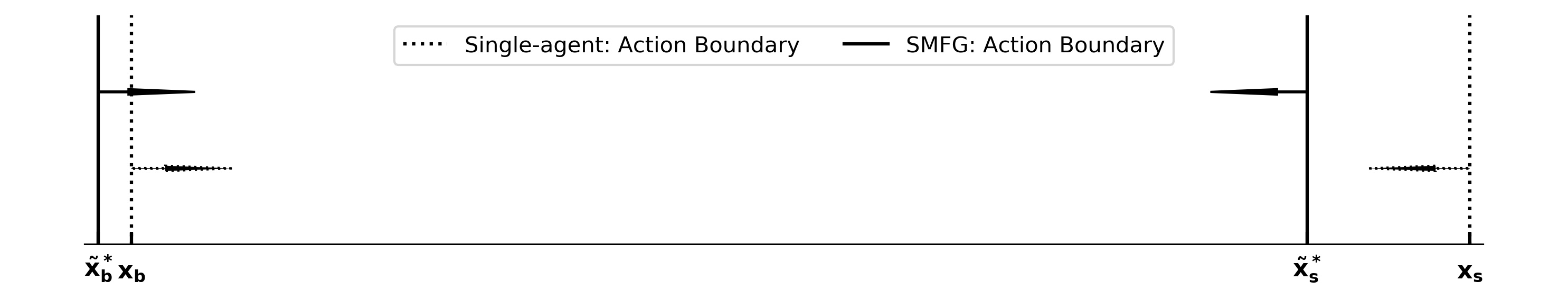}
\caption{Single-agent v.s. MFG}
\label{fig:1vm}
\end{figure}

\paragraph{Impact of $\lambda$.} $\lambda\in(0,1)$ measures the irreversibility of the investment: the closer $\lambda$ to 1, the more irreversible the investment. For the single-agent control problem \eqref{eq:ctrl} (Figures \ref{fig: 1vsm-l-lbd} and \ref{fig: 1vsm-u-lbd}), the expansion threshold $x_b$ stays relatively insensitive with respect to an increasing $\lambda$, the contraction threshold $x_s$ however increases dramatically along with $\lambda$. This means that for an individual company, if the investment is more irreversible, it becomes less profitable to frequently decrease the production level; consequently, the contraction threshold is raised to a higher level. Under the game \eqref{example2} setting, the irreversibility does not have an immediate impact on the optimal strategies (Figure \ref{fig: bdry-lbd}); instead, it drives down the equilibrium price (Figure \ref{fig: rho-lbd}). This suggests that as it becomes less profitable to reduce production when $\lambda$ approaches 1, companies in the game \eqref{example2} tend to keep a higher production level and this tendency collectively reduces the price due to the risk-aversion implied by the Cobb-Douglas function.
\begin{figure}[!ht]
\centering
\begin{subfigure}[t]{0.48\textwidth}
\includegraphics[width = \textwidth]{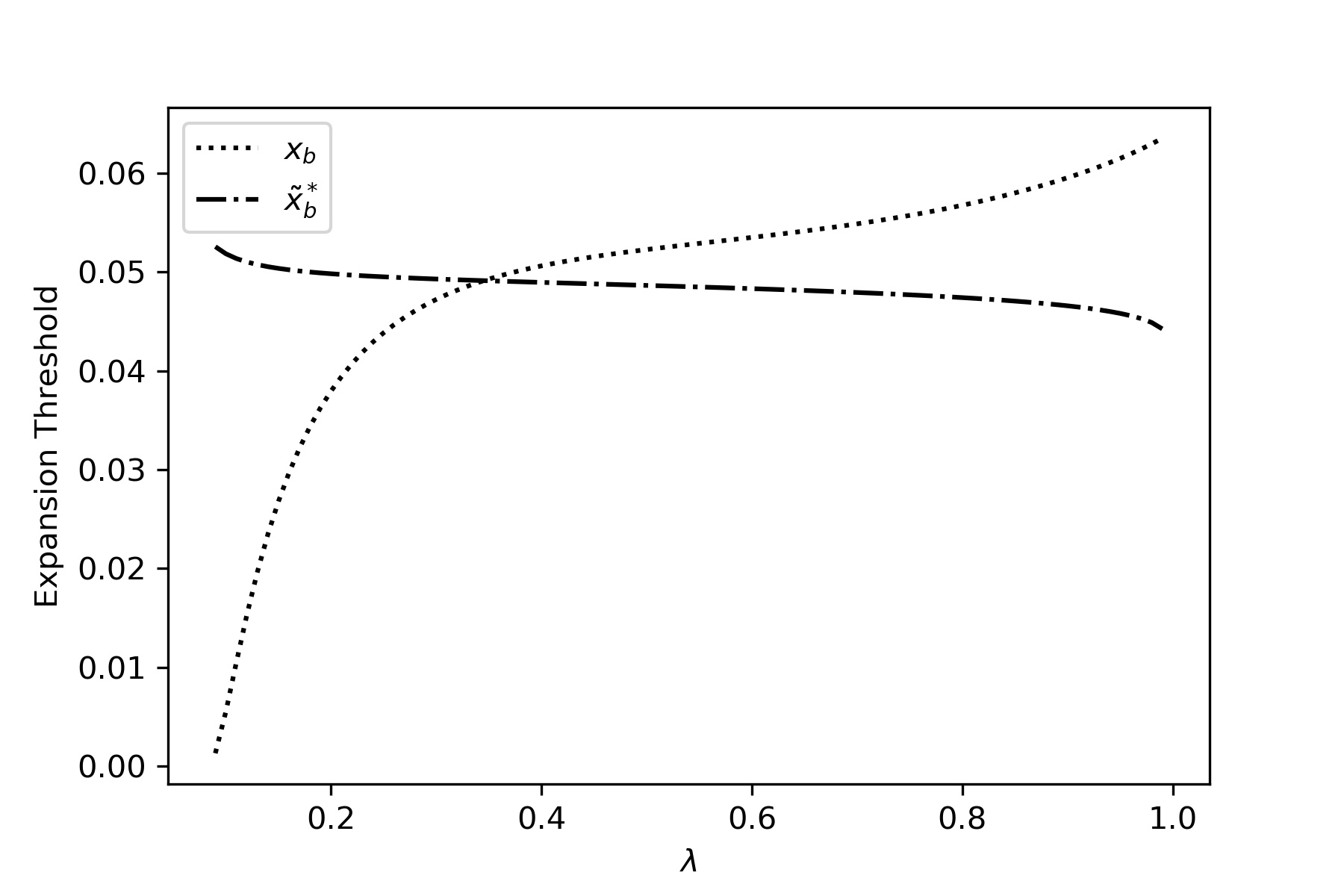}
\caption{Expansion threshold under different values of $\lambda$: single-agent v.s. MFG}
\label{fig: 1vsm-l-lbd}
\end{subfigure}
~
\begin{subfigure}[t]{0.48\textwidth}
\includegraphics[width = \textwidth]{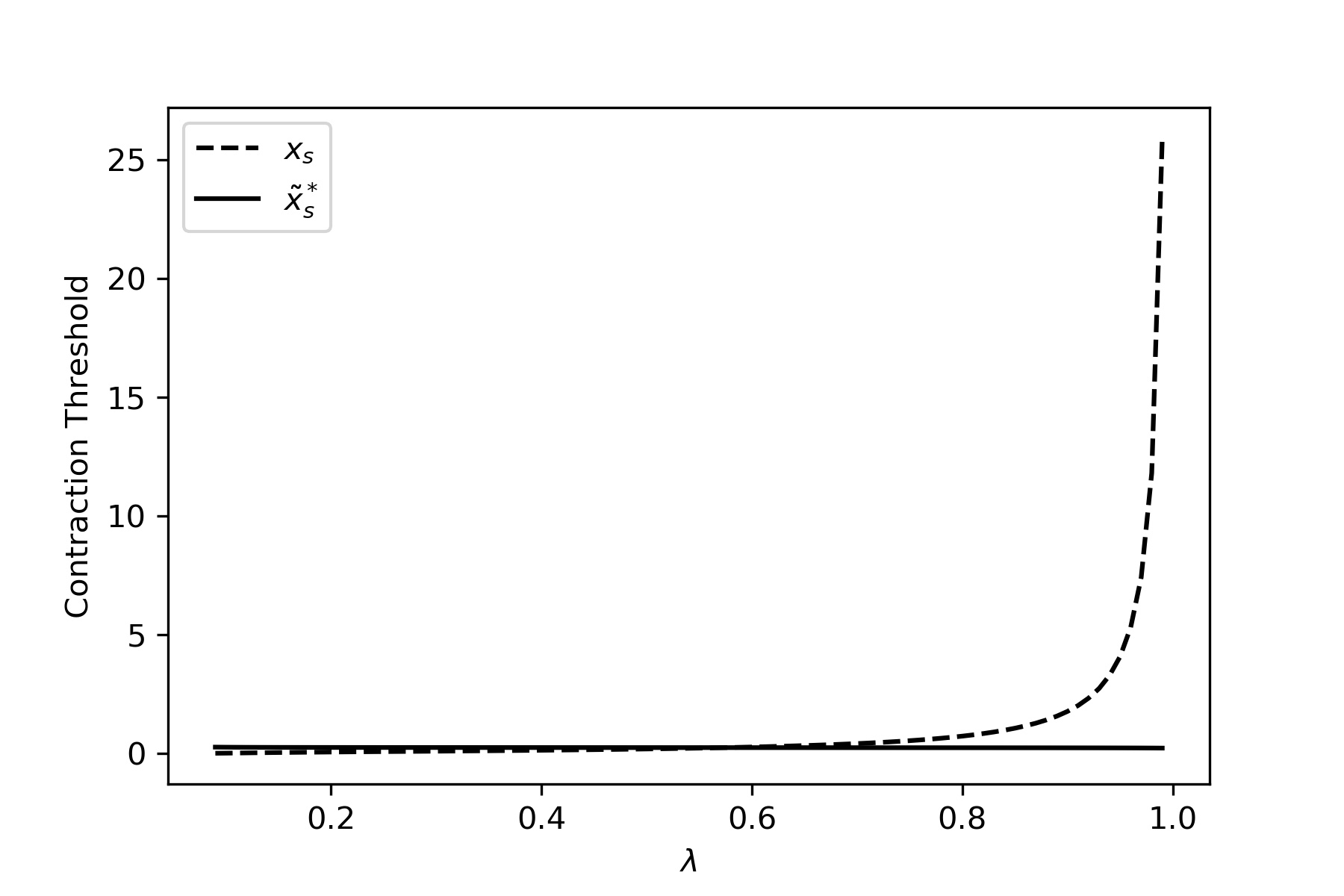}
\caption{Contraction threshold under different values of $\lambda$: single-agent v.s. MFG}
\label{fig: 1vsm-u-lbd}
\end{subfigure}
\\
\begin{subfigure}[t]{0.48\textwidth}
\includegraphics[width = \textwidth]{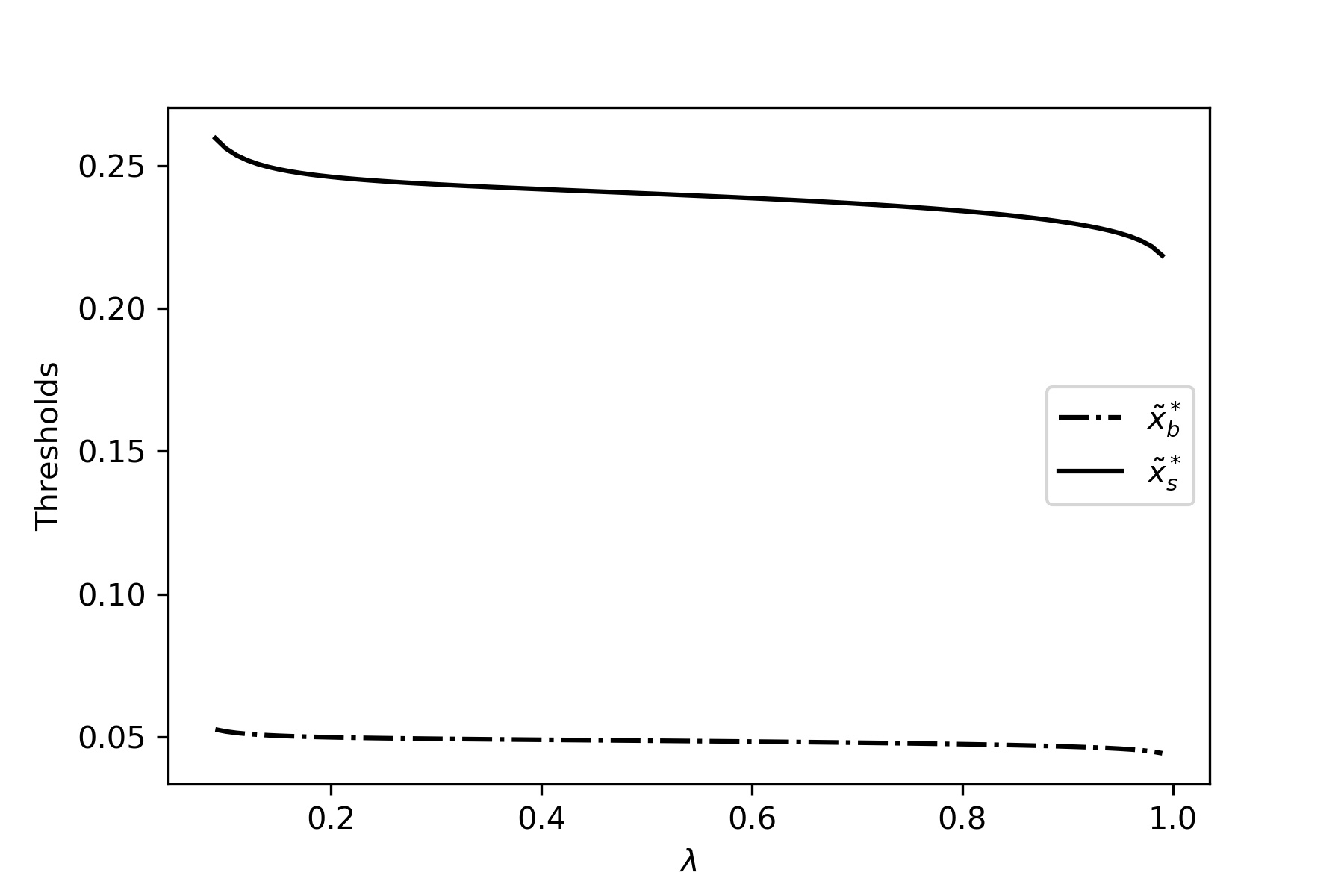}
\caption{MFG optimal thresholds versus $\lambda$}
\label{fig: bdry-lbd}
\end{subfigure}
~
\begin{subfigure}[t]{0.48\textwidth}
\includegraphics[width = \textwidth]{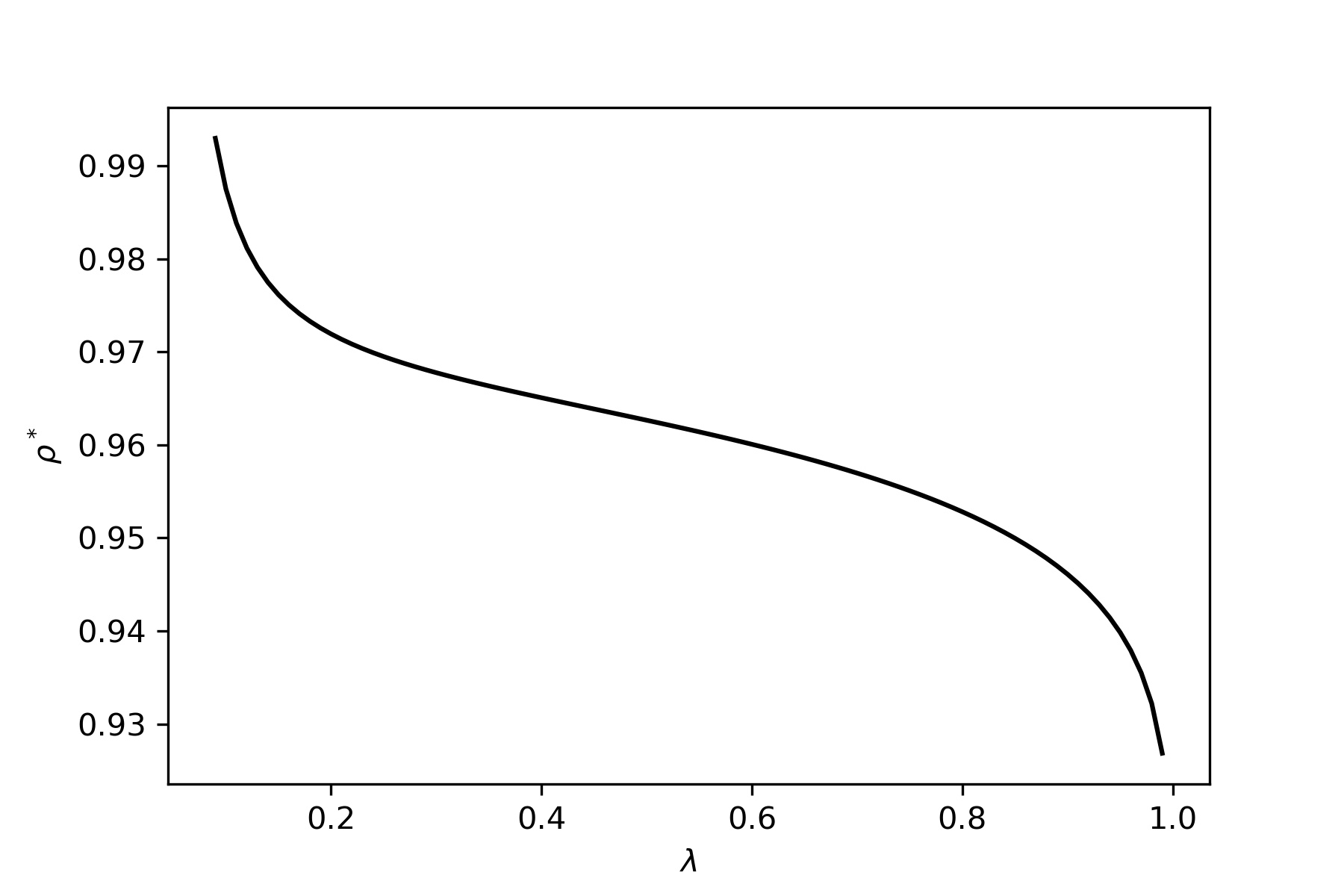}
\caption{Equilibrium price versus $\lambda$}
\label{fig: rho-lbd}
\end{subfigure}
\caption{Impact of $\lambda$.}
\label{fig: lbd}
\end{figure}

\paragraph{Impact of $\delta$ and $\gamma$.} The drift coefficient $\delta$ represents the expected growth rate of the production and $\gamma$ measures the volatility of the growth. The decision of whether or not to adjust the production level is the trade-off between the running payoff $c\rho x_t^\alpha$ and the profit from direct intervention $p(1-\lambda)d\xi^-_t-pd\xi^+_t$, with $\alpha\in(0,1)$. Without any intervention within the time interval $[t,t+\Delta t]$, $x_{t+\Delta t}^\alpha$ is given by 
\begin{equation}\label{eq: dlt}
x_t^\alpha\exp\left\{[\alpha\delta-\frac{\gamma^2}{2}\alpha(1-\alpha)]\Delta t\right\}\exp\left\{\alpha\gamma (W_{t+\Delta t}-W_t)-\frac{\alpha^2\gamma^2}{2}\Delta t\right\},
\end{equation}
 therefore $\alpha\delta-\frac{\gamma^2}{2}\alpha(1-\alpha)$ represents the expected growth rate of $x_t^\alpha$. Under the single-agent setting \eqref{eq:ctrl}, when $\delta$ increases, the revenue function grows faster, leading to higher expansion and contraction thresholds, as shown in Figures \ref{fig: 1vsm-l-dlt} and \ref{fig: 1vsm-u-dlt}. Moreover, the growth in $\delta$ has larger impact on the contraction threshold $x_s$ compared to the the expansion threshold $x_b$. It also implies that each company tends to maintain a higher production level as $\delta$ grows. Under the game \eqref{example2}, this tendency on the individual level is aggregated, driving down the equilibrium price $\rho^*$, as shown in Figure \ref{fig: rho-dlt}. 

\begin{figure}[!ht]
\centering
\begin{subfigure}[t]{0.48\textwidth}
\includegraphics[width = \textwidth]{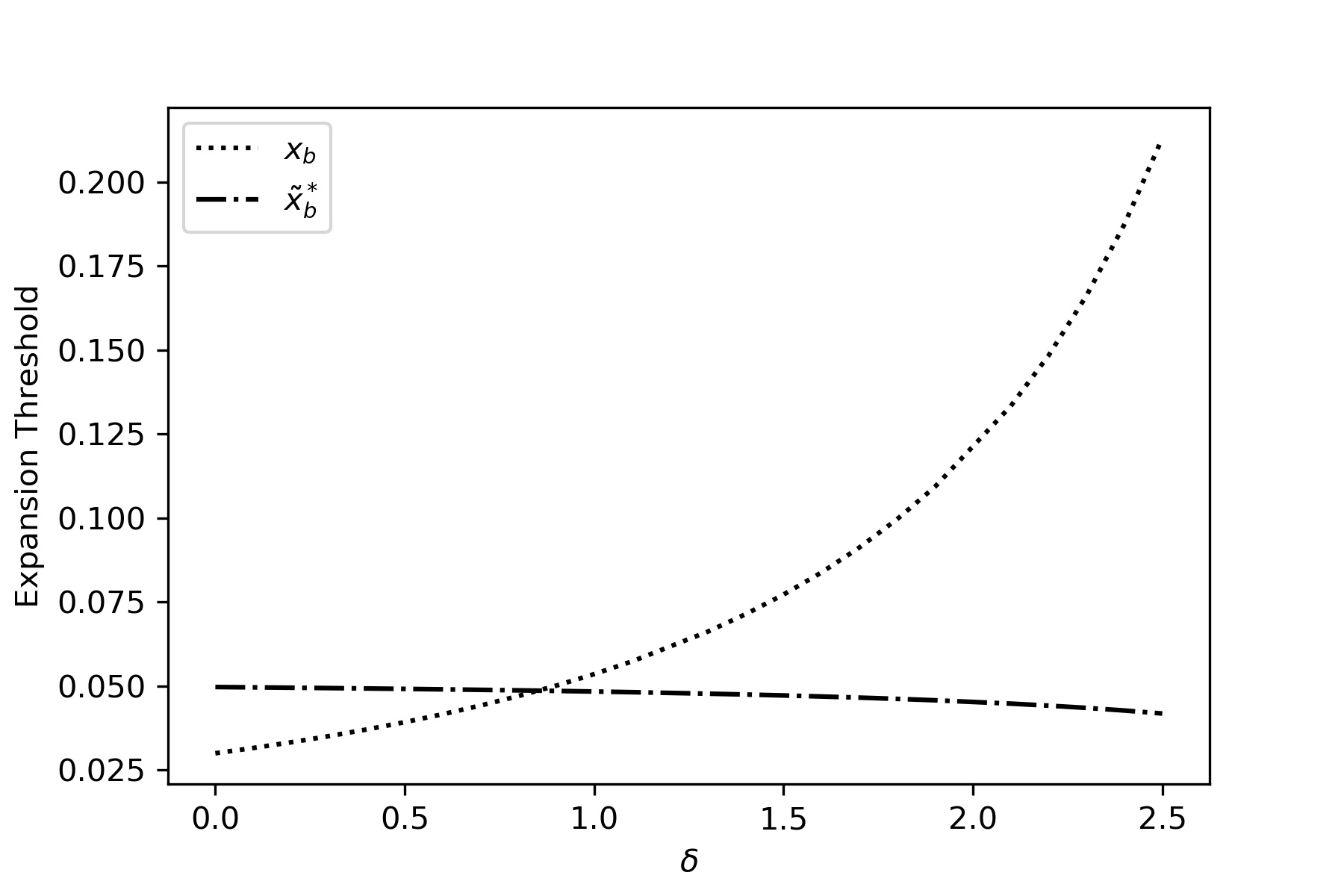}
\caption{Expansion threshold under different values of $\delta$: single-agent v.s. MFG}
\label{fig: 1vsm-l-dlt}
\end{subfigure}
~
\begin{subfigure}[t]{0.48\textwidth}
\includegraphics[width = \textwidth]{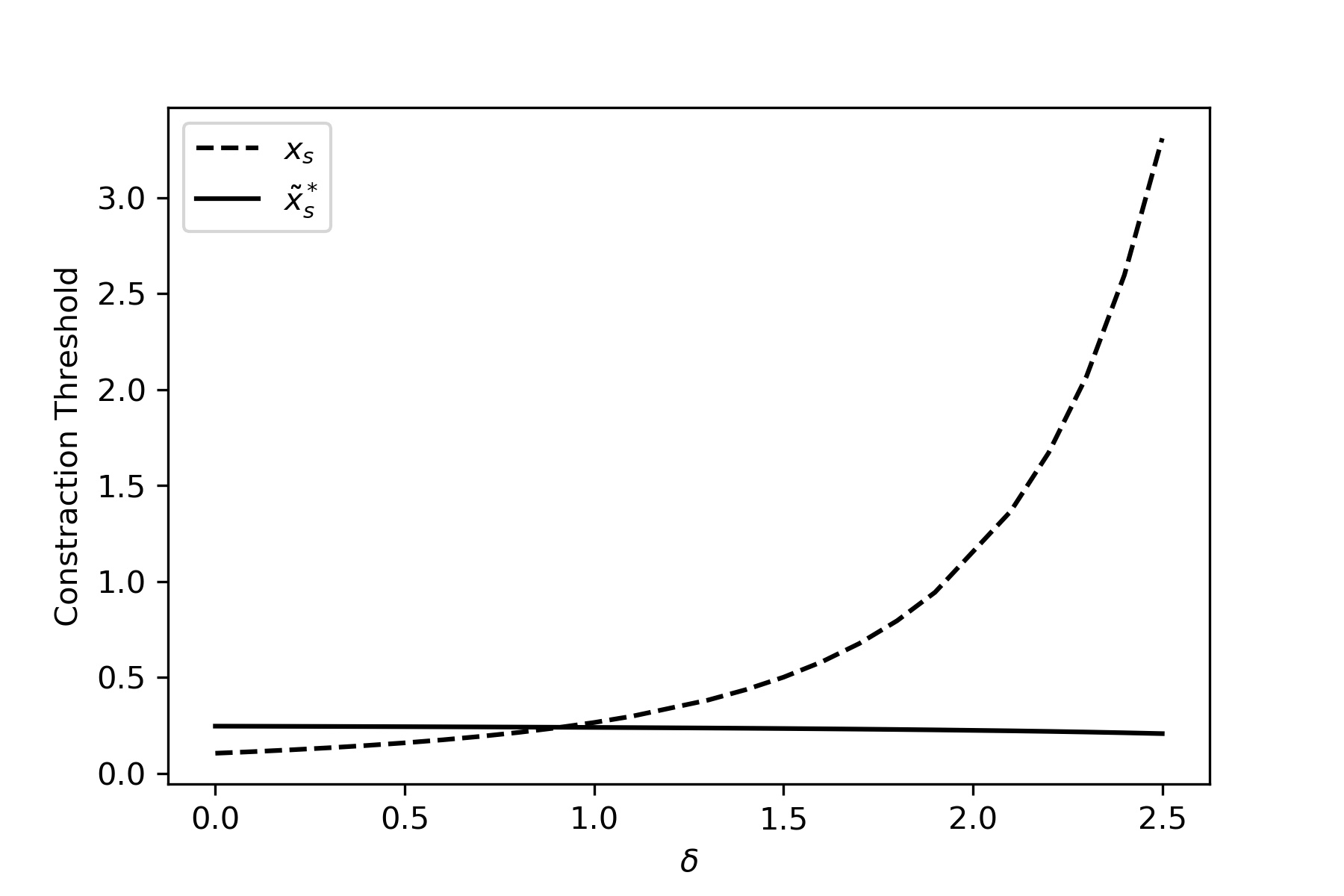}
\caption{Contraction threshold under different values of $\delta$: single-agent v.s. MFG}
\label{fig: 1vsm-u-dlt}
\end{subfigure}
\\
\begin{subfigure}[t]{0.48\textwidth}
\includegraphics[width = \textwidth]{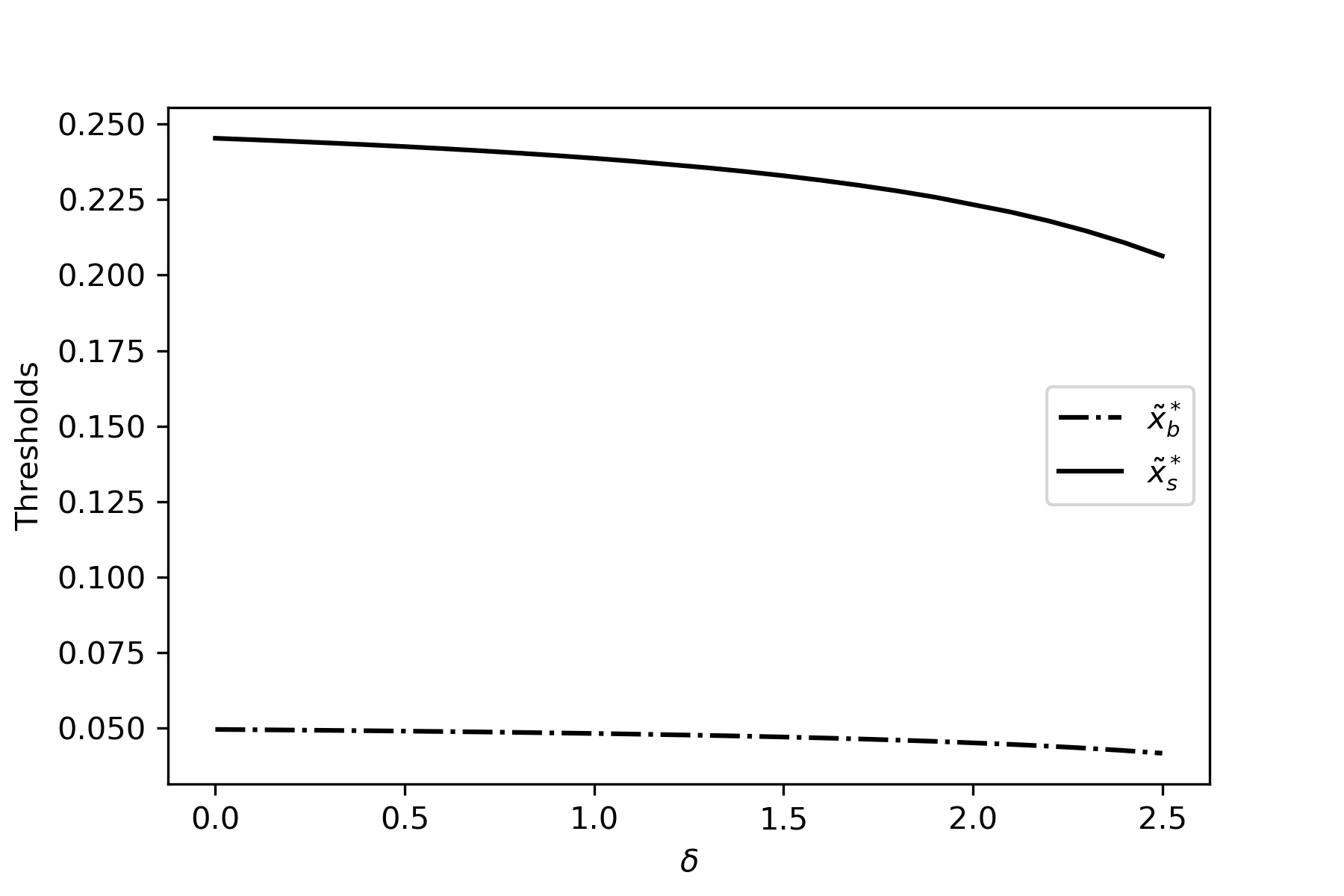}
\caption{MFG optimal thresholds versus $\delta$}
\label{fig: bdry-dlt}
\end{subfigure}
~
\begin{subfigure}[t]{0.48\textwidth}
\includegraphics[width = \textwidth]{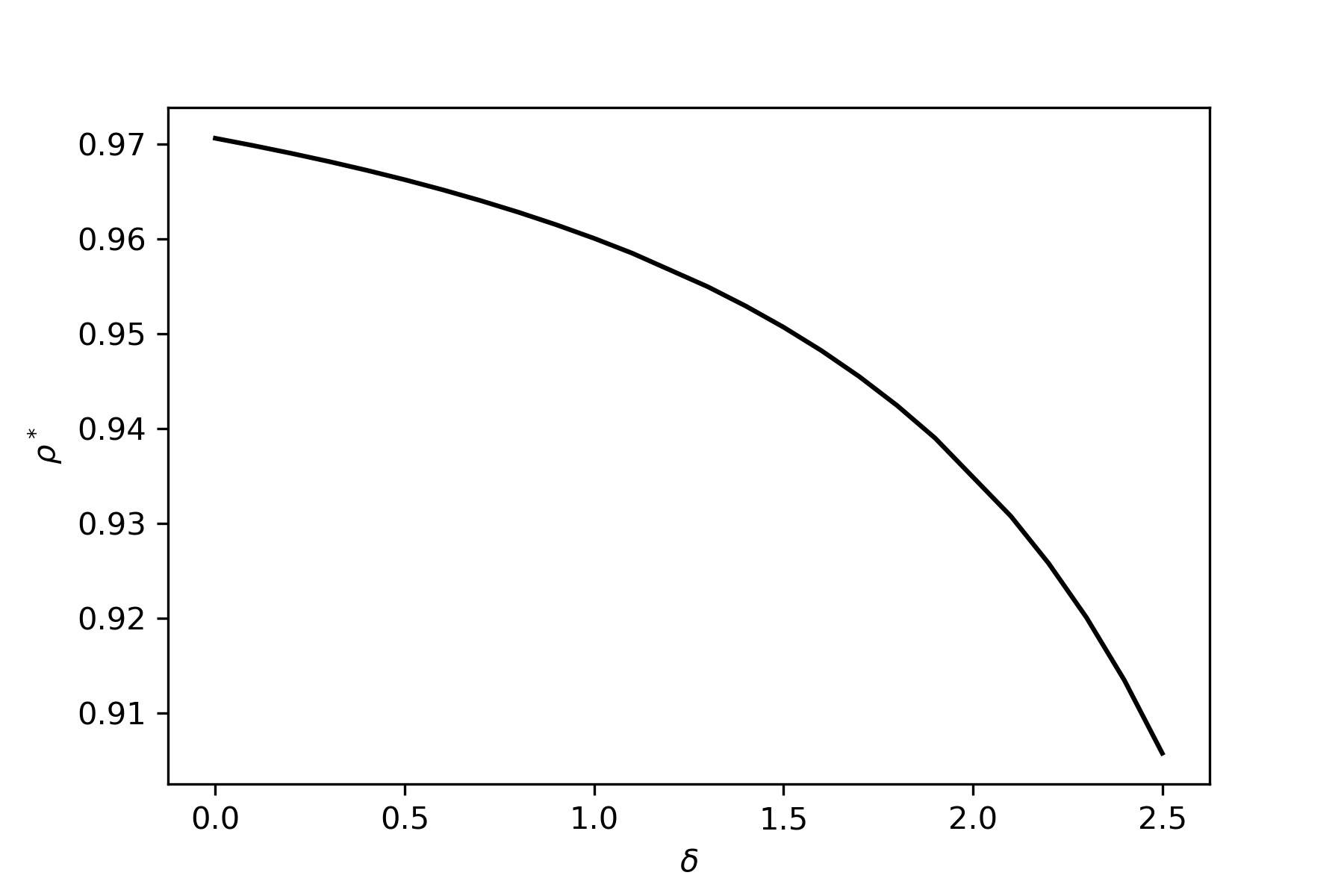}
\caption{Equilibrium price versus $\delta$}
\label{fig: rho-dlt}
\end{subfigure}
\caption{Impact of $\delta$.}
\label{fig: dlt}
\end{figure}

The impact of an increasing $\gamma$ on both the single-agent control problem and the MFG can be seen from the following two perspectives. As $\gamma$ increases, the growth rate of the revenue function $\alpha\delta-\frac{\gamma^2}{2}\alpha(1-\alpha)$ decreases, potentially causing lower expansion and contraction thresholds. An increasing $\gamma$ indicates a larger volatility in the growth rate of the production level and the company can take advantage of the high volatility and reduce the frequency of intervention, potentially decreasing the expansion threshold and increasing the contraction threshold. 

Under both perspectives, the expansion threshold is expected to decrease when $\gamma$ increases. But an increase in $\gamma$ potentially has opposite effects on the contraction threshold. In the single-agent control problem \eqref{eq:ctrl}, the expansion threshold $x_b$ decreases as expected (Figure \ref{fig: 1vsm-l-gm}); the contraction threshold $x_s$ first increases and then decreases (Figure \ref{fig: 1vsm-u-gm}). In the game \eqref{example2}, the prevailing impact of a decreasing growth rate of $x_t^\alpha$ leads to higher the equilibrium price $\rho^*$, as shown in Figure \ref{fig: rho-gm}.

\begin{figure}[!ht]
\centering
\begin{subfigure}[t]{0.48\textwidth}
\includegraphics[width = \textwidth]{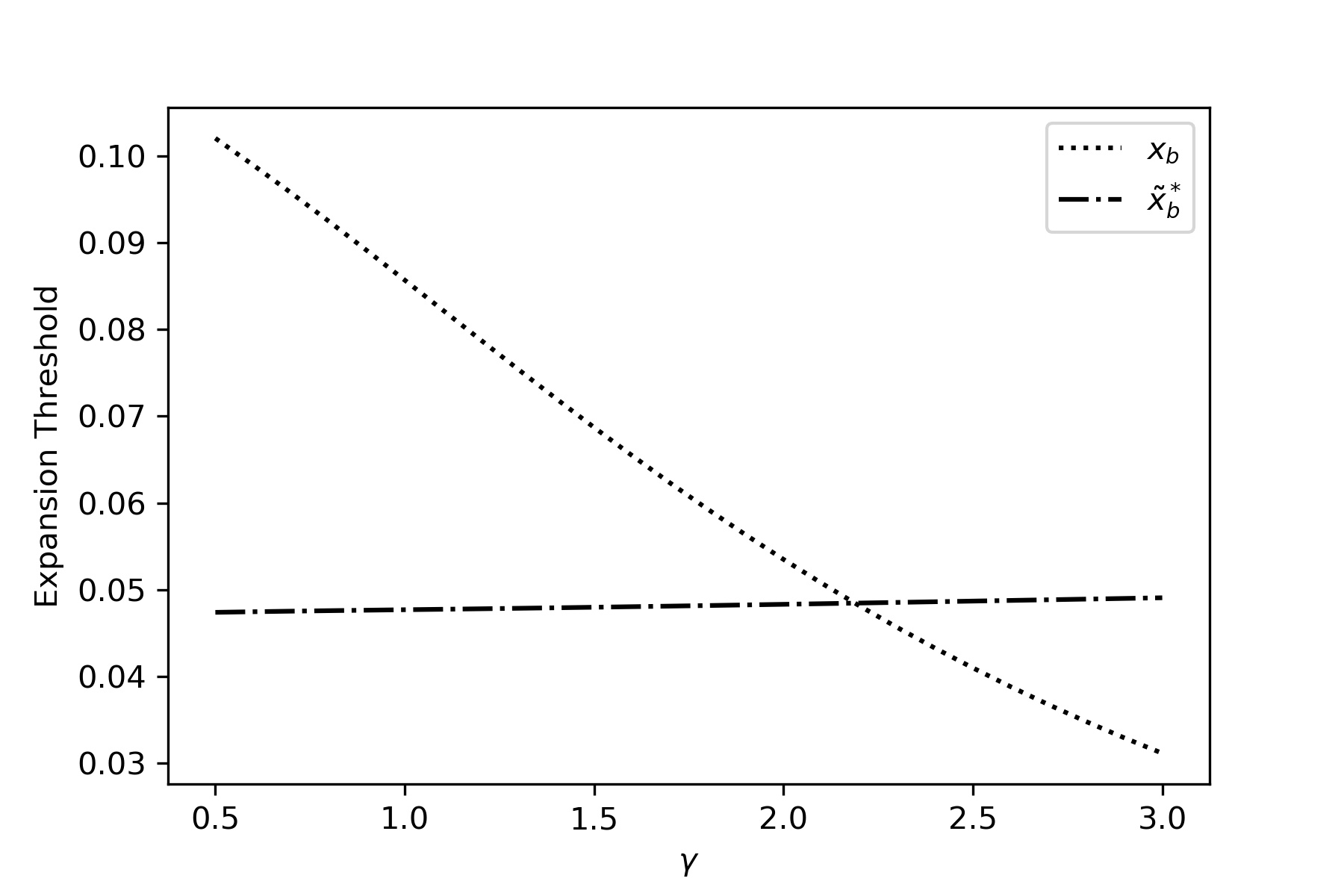}
\caption{Expansion threshold under different values of $\gamma$: single-agent v.s. MFG}
\label{fig: 1vsm-l-gm}
\end{subfigure}
~
\begin{subfigure}[t]{0.48\textwidth}
\includegraphics[width = \textwidth]{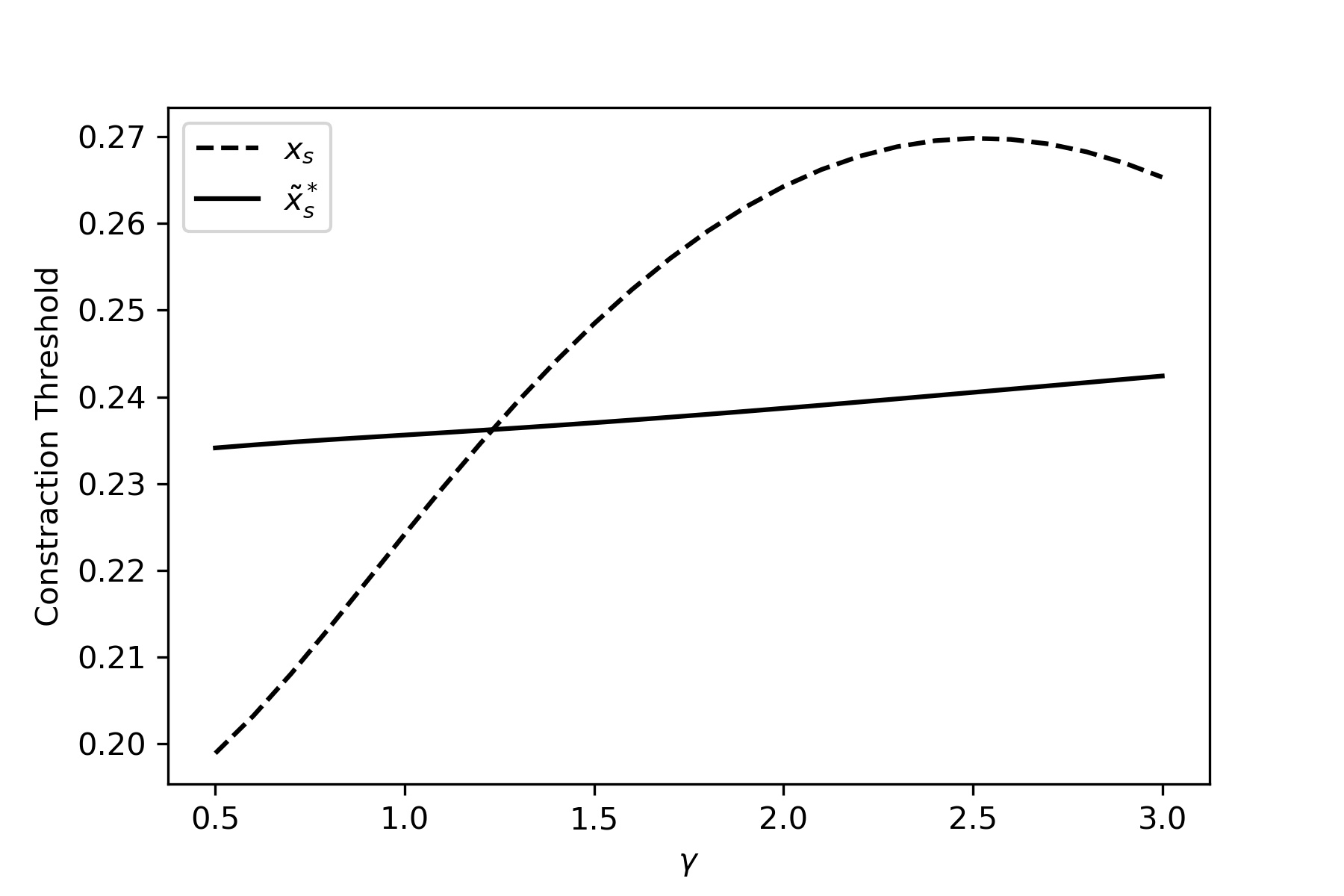}
\caption{Contraction threshold under different values of $\gamma$: single-agent v.s. MFG}
\label{fig: 1vsm-u-gm}
\end{subfigure}
\\
\begin{subfigure}[t]{0.48\textwidth}
\includegraphics[width = \textwidth]{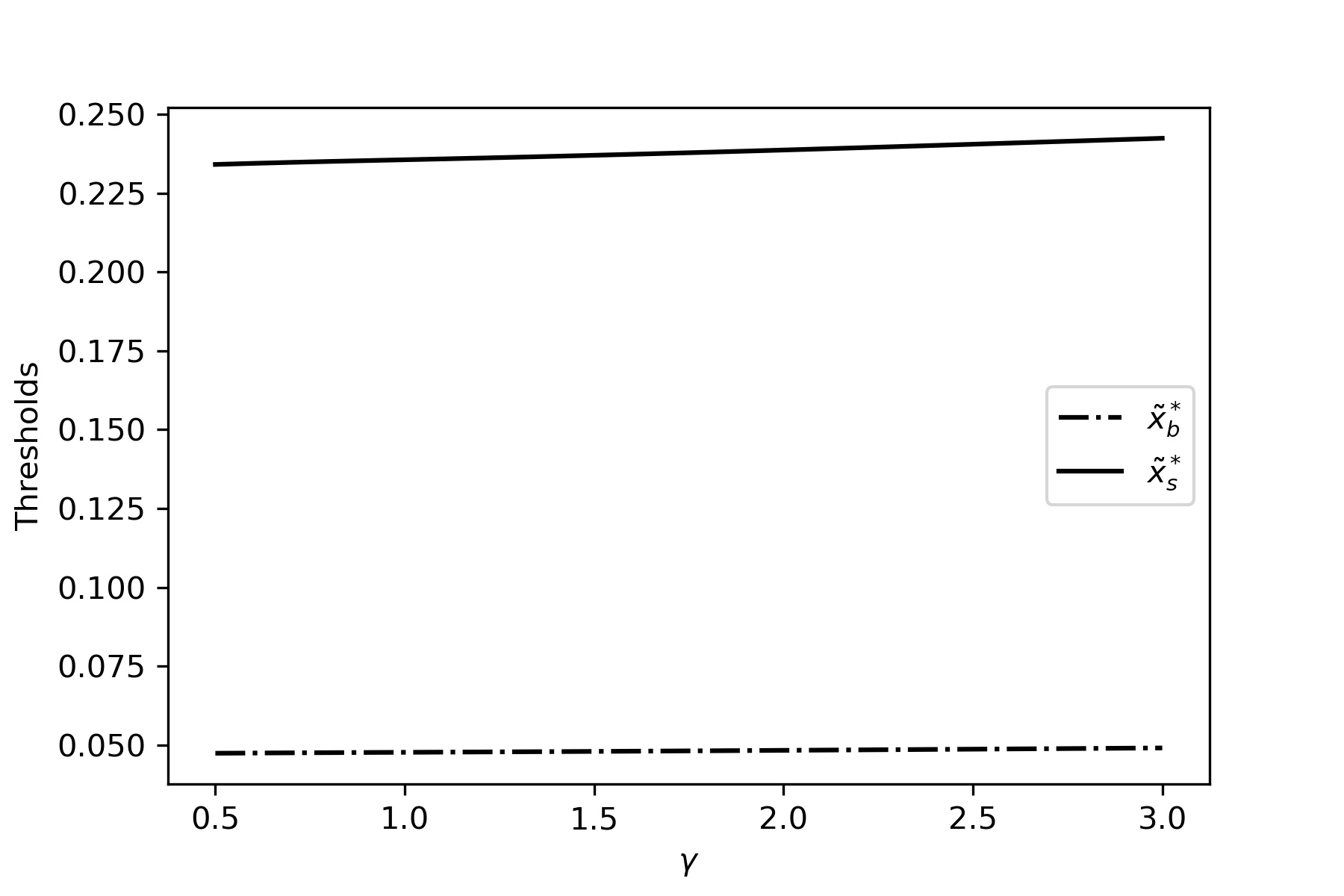}
\caption{MFG optimal thresholds versus $\gamma$}
\label{fig: bdry-gm}
\end{subfigure}
~
\begin{subfigure}[t]{0.48\textwidth}
\includegraphics[width = \textwidth]{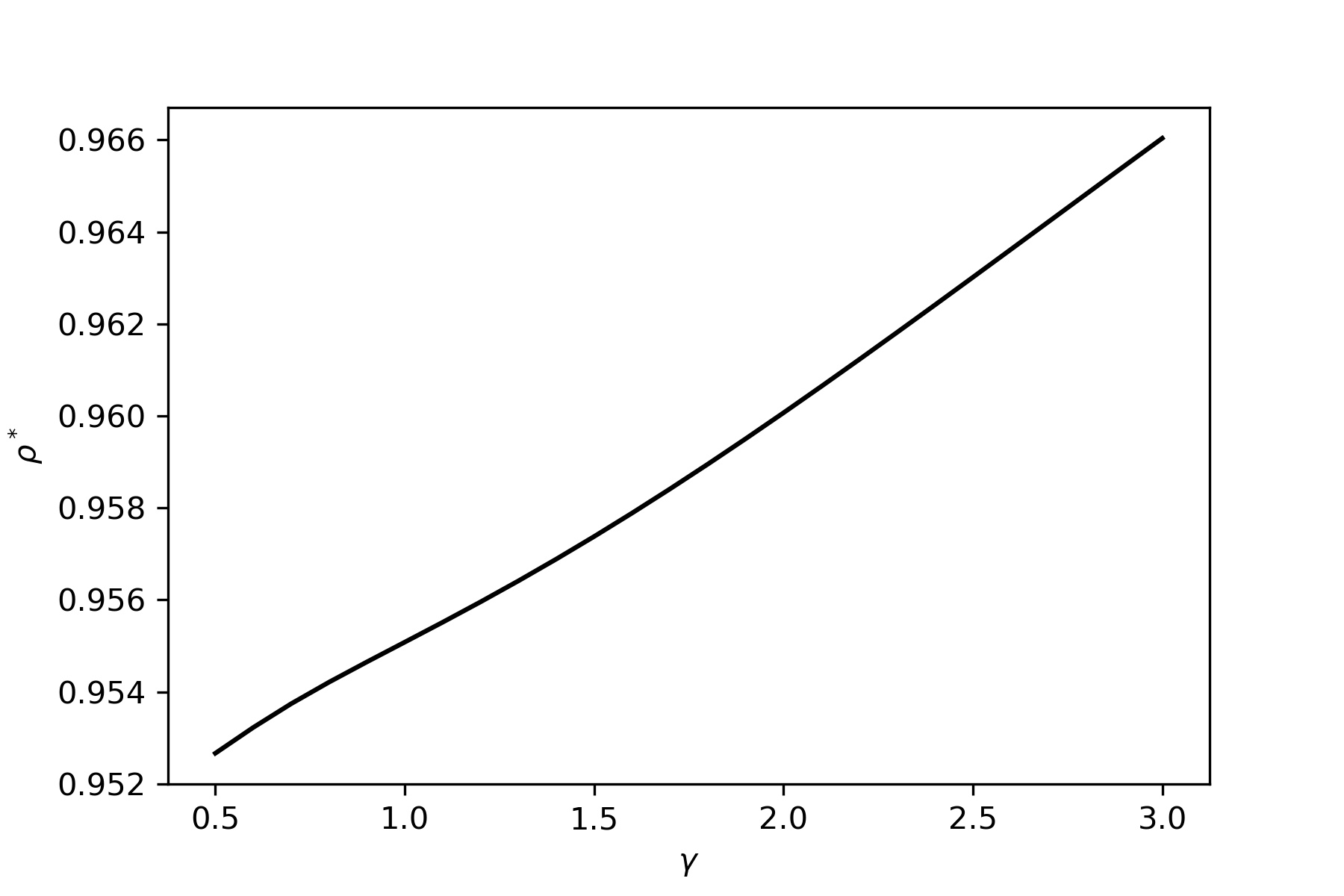}
\caption{Equilibrium price versus $\gamma$}
\label{fig: rho-gm}
\end{subfigure}
\caption{Impact of $\gamma$.}
\label{fig: gm}
\end{figure}

\paragraph{Impact of $r$.} In the single-agent control problem \eqref{eq:ctrl}, as the discount rate $r$ increases, the revenue decays faster as time goes by, thus it becomes more beneficial to decrease the production level for profit and consequently, a significant drop in the contraction threshold in Figure \ref{fig: 1vsm-u-r}. In the game \eqref{example2}, the tendency of decreasing production for each company ultimately drives up the equilibrium price, as shown in Figure \ref{fig: rho-r}.
\begin{figure}[!ht]
\centering
\begin{subfigure}[t]{0.48\textwidth}
\includegraphics[width = \textwidth]{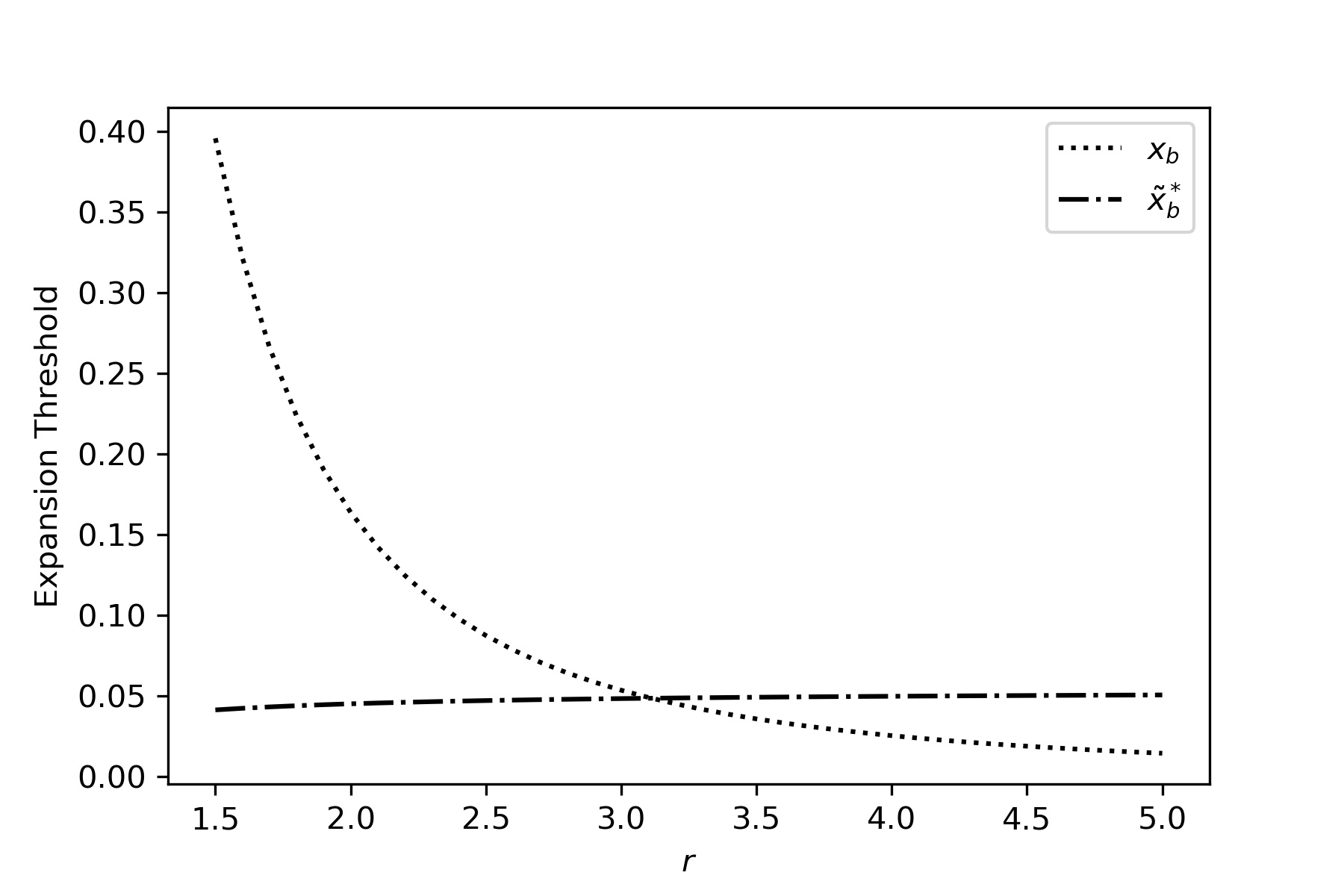}
\caption{Expansion threshold under different values of $r$: single-agent v.s. MFG}
\label{fig: 1vsm-l-r}
\end{subfigure}
~
\begin{subfigure}[t]{0.48\textwidth}
\includegraphics[width = \textwidth]{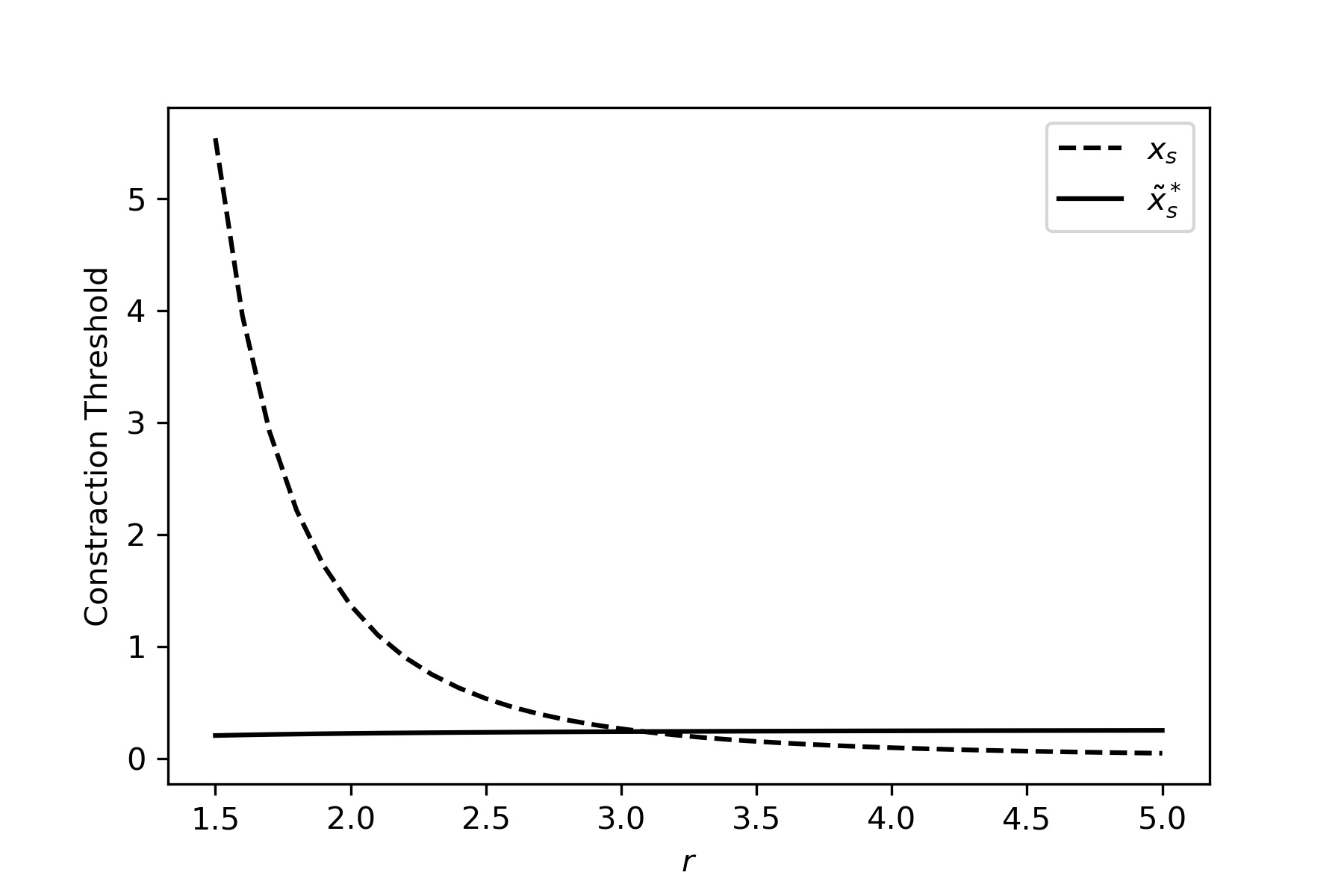}
\caption{Contraction threshold under different values of $r$: single-agent v.s. MFG}
\label{fig: 1vsm-u-r}
\end{subfigure}
\\
\begin{subfigure}[t]{0.48\textwidth}
\includegraphics[width = \textwidth]{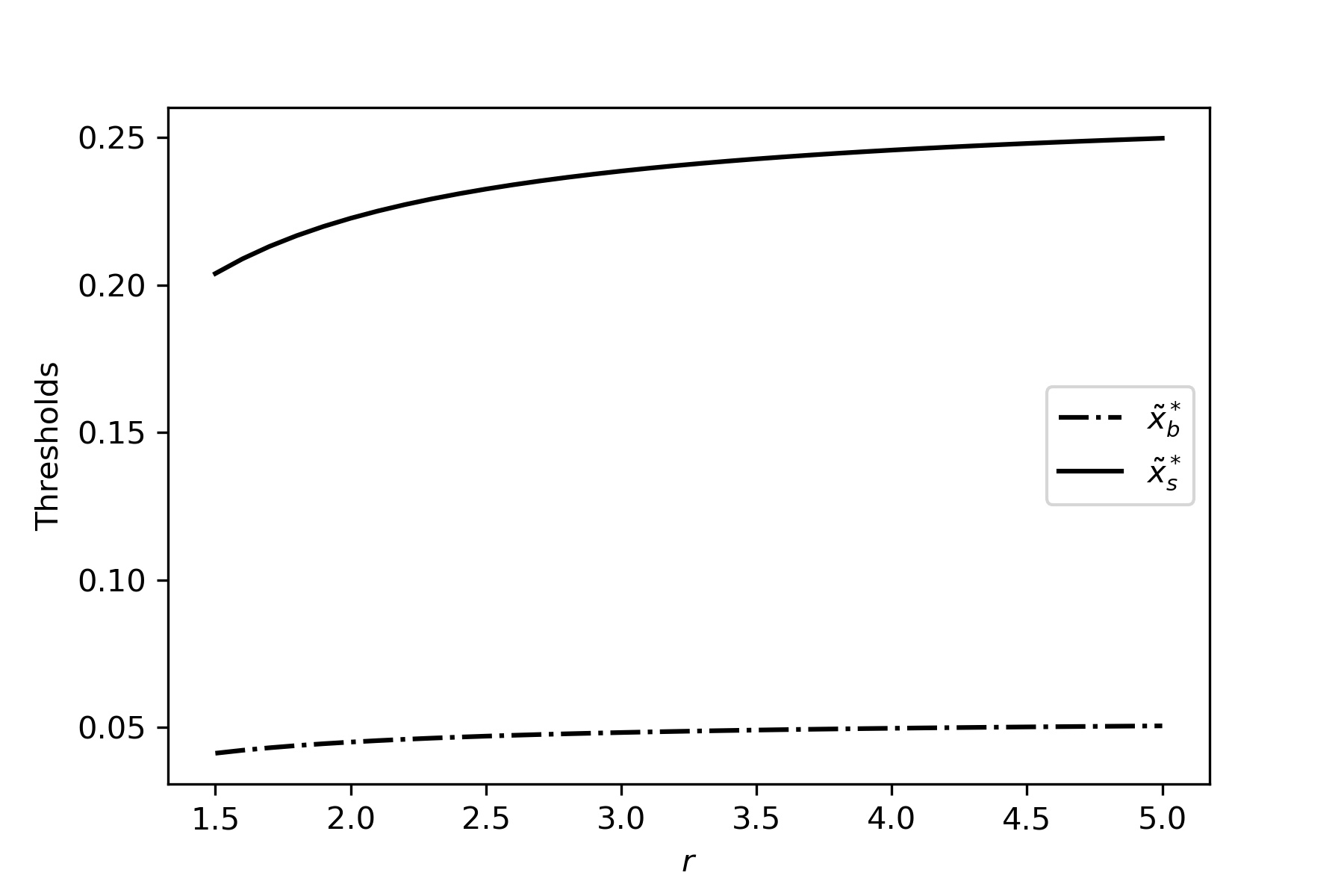}
\caption{MFG optimal thresholds versus $r$}
\label{fig: bdry-r}
\end{subfigure}
~
\begin{subfigure}[t]{0.48\textwidth}
\includegraphics[width = \textwidth]{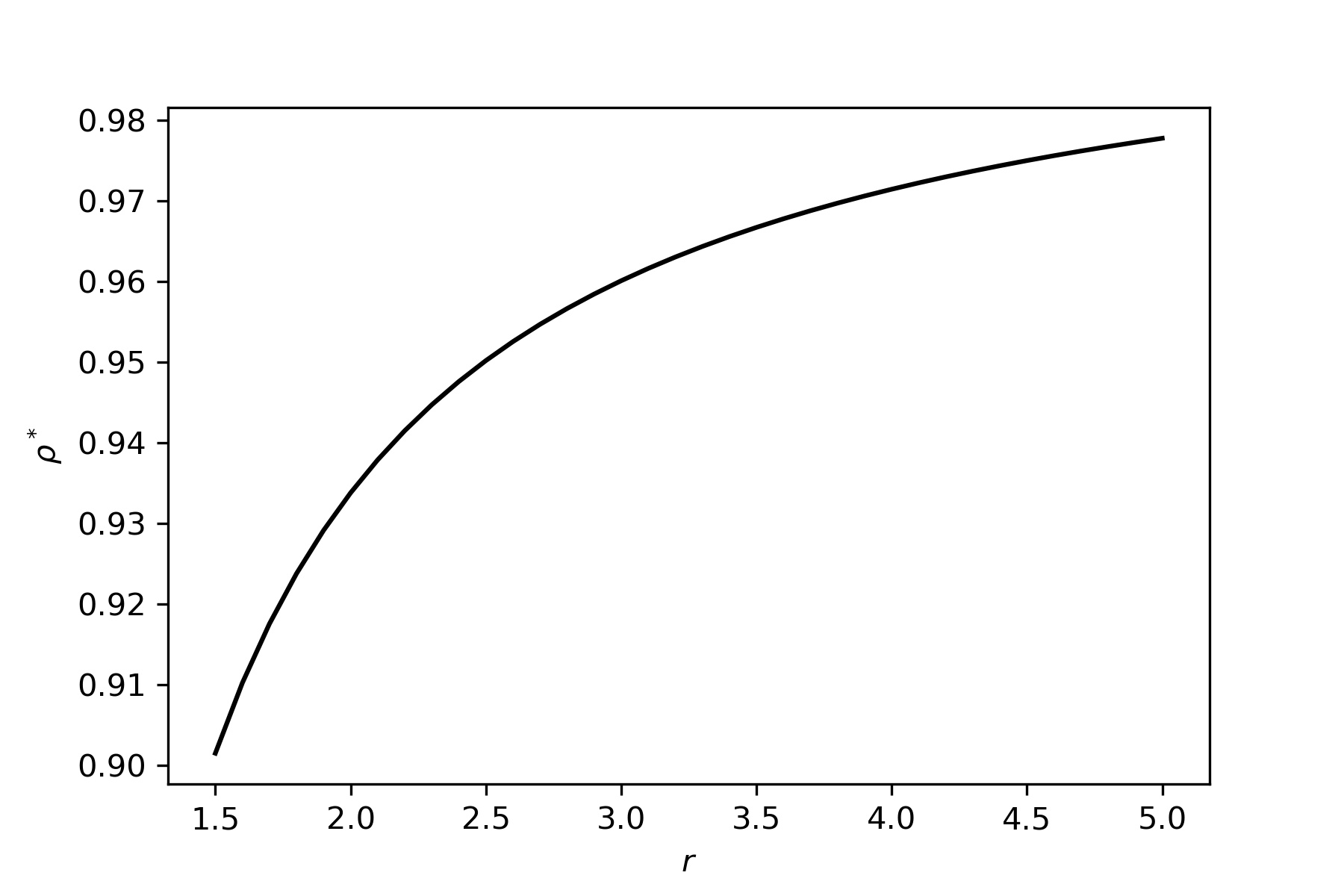}
\caption{Equilibrium price versus $r$}
\label{fig: rho-r}
\end{subfigure}
\caption{Impact of $r$.}
\label{fig: r}
\end{figure}

\paragraph{Impact of $\alpha$.} $\alpha\in(0,1)$ measures the elasticity of the profit with respect to the production. Under the single-agent setting \eqref{eq:ctrl}, both thresholds first increase and then decrease as $\alpha$ approaches 1. In the game \eqref{example2}, the more sensitive the revenue with respect to production, the lower the equilibrium price, as shown in Figure \ref{fig: rho-aph}.
\begin{figure}[!ht]
\centering
\begin{subfigure}[t]{0.48\textwidth}
\includegraphics[width = \textwidth]{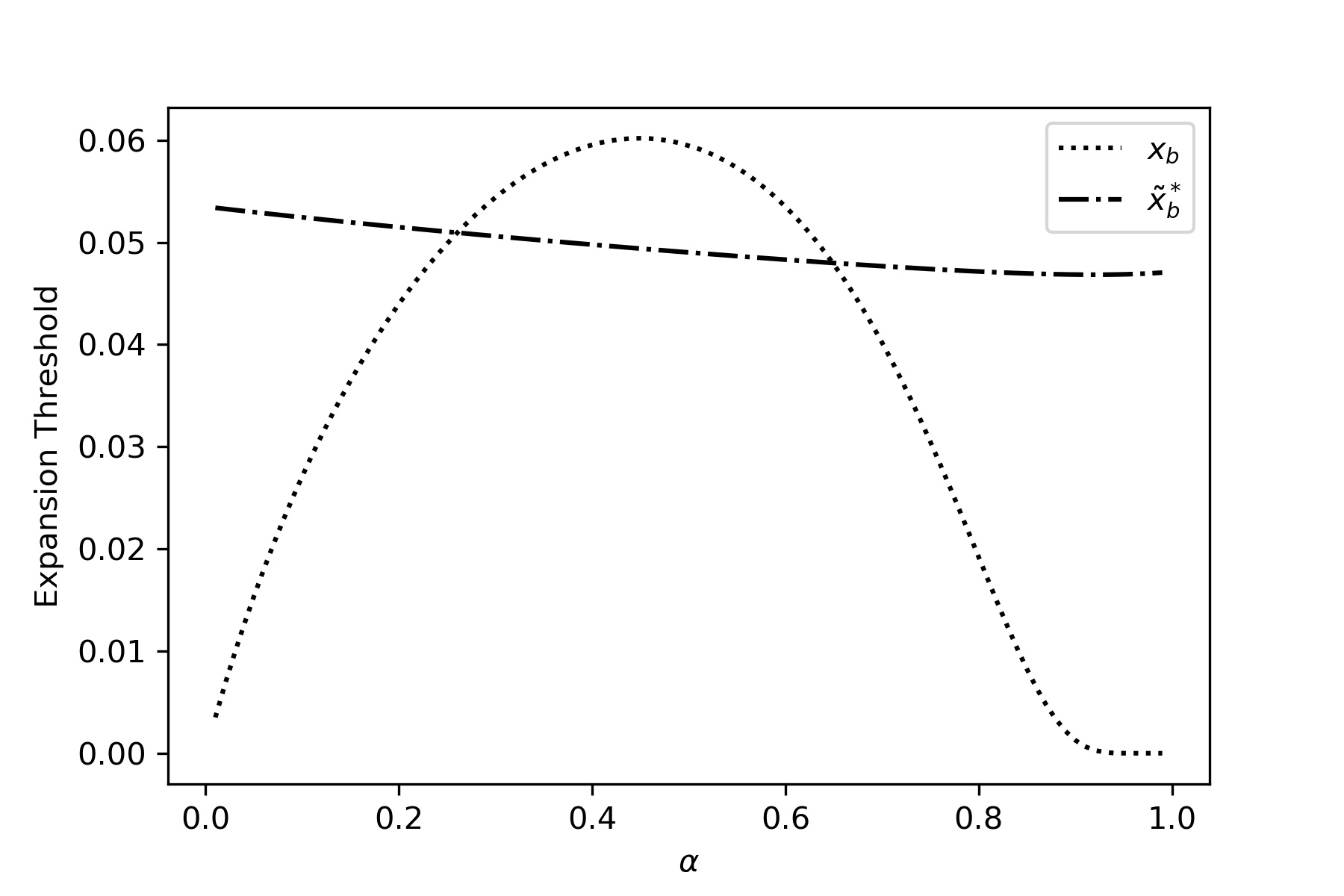}
\caption{Expansion threshold under different values of $\alpha$: single-agent v.s. MFG}
\label{fig: 1vsm-l-aph}
\end{subfigure}
~
\begin{subfigure}[t]{0.48\textwidth}
\includegraphics[width = \textwidth]{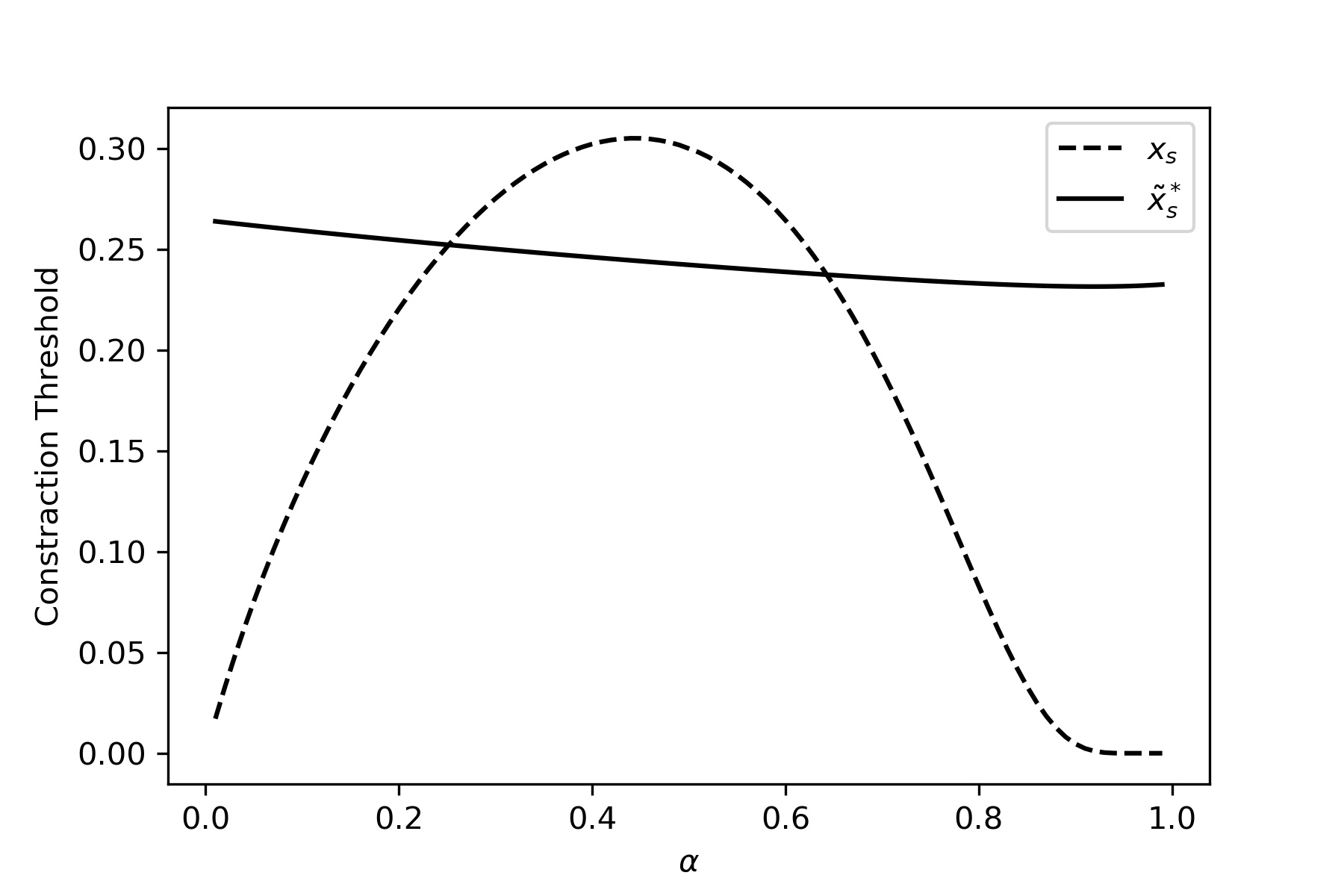}
\caption{Contraction threshold under different values of $\alpha$: single-agent v.s. MFG}
\label{fig: 1vsm-u-aph}
\end{subfigure}
\\
\begin{subfigure}[t]{0.48\textwidth}
\includegraphics[width = \textwidth]{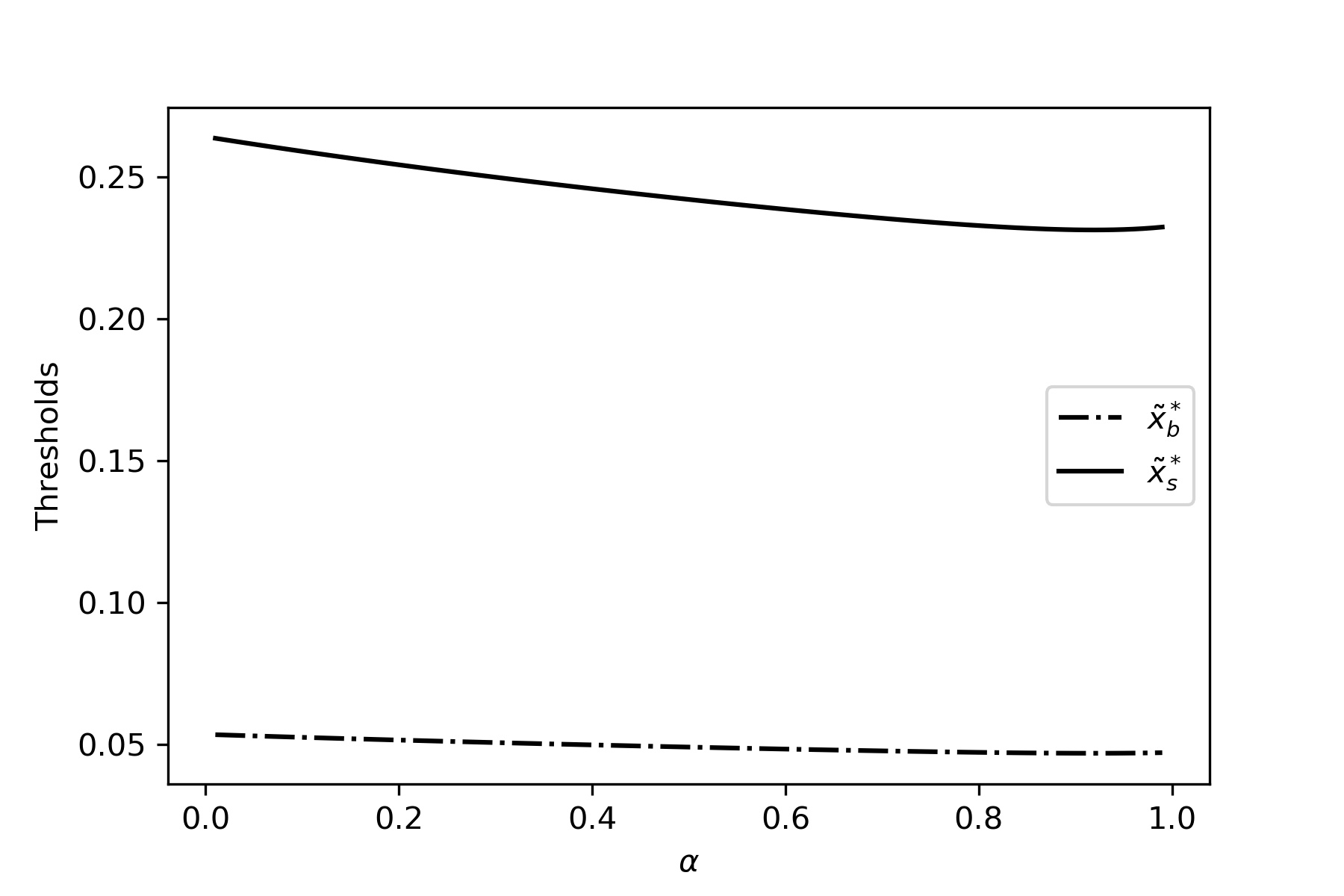}
\caption{MFG optimal thresholds versus $\alpha$}
\label{fig: bdry-aph}
\end{subfigure}
~
\begin{subfigure}[t]{0.48\textwidth}
\includegraphics[width = \textwidth]{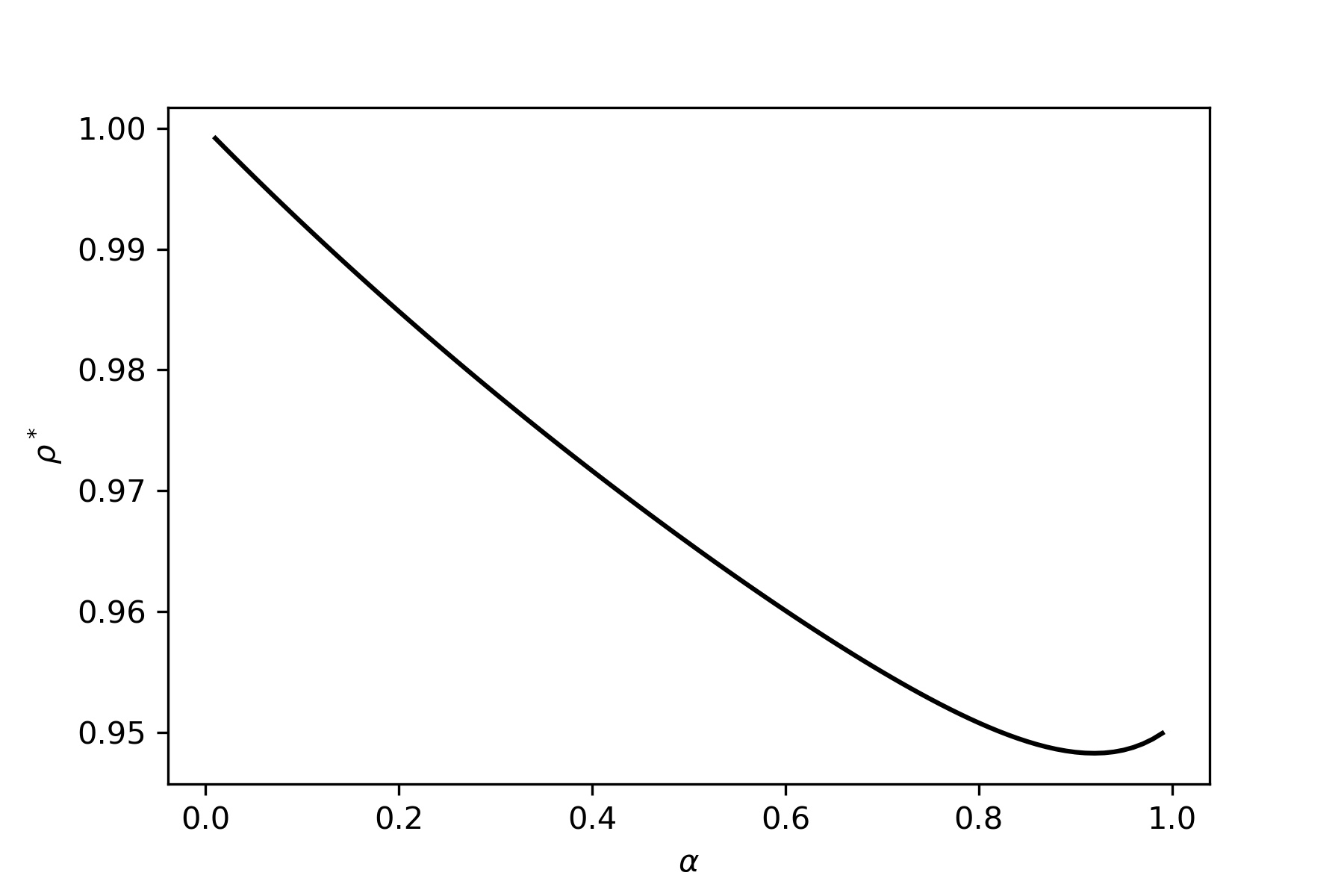}
\caption{Equilibrium price versus $\alpha$}
\label{fig: rho-aph}
\end{subfigure}
\caption{Impact of $\alpha$.}
\label{fig: aph}
\end{figure}

\section{Approximation of $N$-player game}\label{sec:approx}
In Section \ref{ssec: comparison}, we compare the solution to game \eqref{example2} with the solution to the single-agent control problem \eqref{eq:ctrl}, and demonstrate the game effect by analyzing the impact of the model parameters. In this section, we will show that the game \eqref{example2} is an approximation of its associated  $N$-player game, in the sense of $\epsilon$-NE.

Take the filtered probability space $(\Omega, \mf,\mathbb F=\{\mf_t\}_{t\geq0}, \p)$ that supports a standard Brownian motion $W=\{W_t\}_{t\geq0}$. Take $N$ identical copies of the Brownian motion $W$, $W^i=\{W^i_t\}_{t\geq0}$ with $i=1,\dots, N$, such that $W^i$'s are i.i.d. and independent of $W$.

Suppose there are $N$ companies participating in the game of partially reversible investment. For each company $i$, denote $x^i=\{x^i_t\}_{t\geq0}$ as its production level on $\rn$, with initial states $x^i_{0-}\overset{i.i.d.}{\sim}\mu_0\in\mathcal P^2(\rn)$. Similar to \eqref{eq:ad-ctrl-set-1} and considering Remark \ref{rmk:l2}, define the set of admissible controls $\mathcal U^N$ for each company,
\begin{equation}\label{eq:n-adm-ctrl-set}
	\begin{aligned}
		\mathcal{U}^N=&\left\{(\xi^+_\cdot,\xi^-_\cdot):\xi^\pm_\cdot\text{ adapted to }\mathbb F^{(W^1,\dots,W^N)},\,\text{nondecreasing, c\`adl\`ag,}\right.\\
		&\hspace{50pt}\left.\xi^+_{0-}=\xi^-_{0-}=0,\, \eee\left[\int_0^\infty e^{-rt}d\xi^+_t\right]<\infty, \right.\\
		&\hspace{50pt}\left.\text{controlled process }x_t\geq0,\,\forall t\geq0,\,\eee\left[\int_0^\infty e^{-rt}x_t^2dt\right]<\infty\right\},
	\end{aligned}
\end{equation}
where $\mathbb F^{(W^1,\dots,W^N)}=\{\mathcal F^{(W^1,\dots,W^N)}_t\}_{t\geq0}$ is the filtration generated by $(W^1,\dots, W^N)$. For any $\xi^i_\cdot=(\xi^{i,+}_\cdot,\xi^{i,-}_\cdot)\in\mathcal U^N$, assume that the process $x^i=\{x^i_t\}_{t\geq0}$ is driven by
\begin{equation}
    \label{eq:n-dyn}
    dx^{i}_t=x^i_t(\delta dt+\gamma dW^i_t)+d\xi^{i,+}_t-d\xi^{i,-}_t,\quad x^i_{0-}\sim \mu_0.
\end{equation}

Here we consider a similar payoff function for each individual company as in problem \eqref{eq:ctrl}. However, unlike \eqref{eq:ctrl} where the price in the revenue function is exogenously given, here  in the $N$-player game $\rho^i$ the price for company $i$ is assumed to depend on the average of all its opponents' limiting product levels $\frac{\sum_{j-1}x^j_\infty}{N-1}$, and the price is assumes to be determined by the inverse demand function
\[\tilde\rho(x)=a_0-a_1x^{1-\alpha},\]
where $a_0,a_1$ are some positive constants with $a_1$ satisfying \eqref{eq:a1}. 

Under a given set of other companies' controls, $\bfxi^{-i}_\cdot=(\xi^1_\cdot,\dots,\xi^{i-1}_\cdot,\xi^{i+1}_\cdot,\dots,\xi^N_\cdot)$, for any $\bfx\in\rn^N$, the payoff function for company $i$, is given by
\begin{equation}\label{eq:n-payoff}
J^i(\bfx,\xi^i_{\cdot};\bfxi^{-i}_\cdot)=\eee\left[\int_0^\infty e^{-rt}\left[\frac{c(x^i_t)^{\alpha}}{N-1}\sum_{j\neq i}\tilde\rho(x^{j}_\infty)dt-pd\xi^{i,+}_t+p(1-\lambda)d\xi^{i,-}_t\right]\biggl|\bfx_{0-}=\bfx\right],
\end{equation}
where $\bfx_{0-}=(x^1_{0-},\dots,x^N_{0-})$. The objective of company $i$ is to choose the best control policy $\xi^{*,i}\in\mathcal U^N$ to maximize the above payoff. That is,
\begin{equation}\label{eq:n-fv}\tag{N-player}
	\sup_{\xi^{i,+}_\cdot,\xi^{i,-}_\cdot\in\mathcal U^N}J^i(\bfx,\xi^i_\cdot;\bfxi^{-i}_\cdot)
\end{equation}
subject to \eqref{eq:n-dyn}.

There are various solution criteria for an $N$-player game. In this section, we focus on the notion of the Nash equilibrium (NE). 
An NE of an $N$-player game is a set of strategies of all agents from which no players has the incentive to unilaterally deviate. More specifically, 
\begin{defn}[NE]\label{defn:nash}
$\bfxi^*_\cdot=(\xi^{*,1}_\cdot,\dots,\xi^{*,N}_\cdot)$ is called an NE to the game \eqref{eq:n-fv} if for any $i=1,\dots,N$,
\[\eee_{\mu_0}\left[J^i(\bfx_{0-},\xi^{*,i}_\cdot;\bfxi^{*,-i}_\cdot)\right]\geq\eee_{\mu_0}\left[J^i(\bfx_{0-},\xi^{i}_\cdot;\bfxi^{*,-i}_\cdot)\right],\quad \forall \xi^i_\cdot\in\mathcal U^N,\]
where $x^k_{0-}\overset{i.i.d.}{\sim}\mu_0$, $k=1,\dots,N$.
\end{defn}
Solving for such an NE analytically is challenging especially when $N$ is large. We will show that the solution for the game \eqref{example2} in Section \eqref{ssec:exp-soln} provides an approximation of the game \eqref{eq:n-fv} in the following sense.
\begin{defn}[$\epsilon$-NE]\label{defn:eps-nash}
For some $\epsilon>0$, $\bfxi^*_\cdot=(\xi^{*,1}_\cdot,\dots,\xi^{*,N}_\cdot)$ is called an $\epsilon$-NE to the game \eqref{eq:n-fv} if for any $i=1,\dots,N$,
\[\eee_{\mu_0}\left[J^i(\bfx_{0-},\xi^{*,i}_\cdot;\bfxi^{*,-i}_\cdot)\right]\geq\eee_{\mu_0}\left[J^i(\bfx_{0-},\xi^{i}_\cdot;\bfxi^{*,-i}_\cdot)\right]-\epsilon,\quad \forall \xi^i_\cdot\in\mathcal U^N,\]
where $x^k_{0-}\overset{i.i.d.}{\sim}\mu_0$, $k=1,\dots,N$.
\end{defn}

To see the approximation, first recall the definition of a solution $(\xi^{*}_\cdot,\rho^*)$ to the \eqref{example2} given by Definition \ref{soln-smfg} and its explicit form given in Section \ref{ssec:exp-soln} characterized by the pair of reflection boundaries and mean information $(\tilde x^*_b, \tilde x^*_s,\rho^*)$. 

Now, for any company $k=1,\dots,N$, consider the following admissible control policy $\bar\xi^k_\cdot$ characterized by  the reflection boundaries $(\tilde x^*_b, \tilde x^*_s)$ such that for the controlled process $\bar x^k=\{\bar x^k_t\}_{t\geq0}$, $\bar x^k_t\in[\tilde x^*_b, \tilde x^*_s]$ for almost all $t\geq0$. Fix a representative company $i$. Suppose that for any $j\neq i$, company $j$ decides to take the policy $\bar\xi^j_\cdot\in\mathcal U^N$. Then $\bar x^j$'s are i.i.d.; moreover, according to Definition \ref{soln-smfg}, the consistency condition of the solution to the game \eqref{example2} guarantees that $\rho^*=\eee[\tilde\rho(\bar x^j_\infty)]$.

Now denote the set of strategies consisting of $\bar\xi^k_\cdot$'s by the following vector
\begin{equation}\label{eqn:xi}
    \bar\bfxi_\cdot=(\bar\xi^1_\cdot,\dots,\bar\xi^N_\cdot).
    \end{equation}
Then we have
\begin{thm} For any fixed $N$, $\bar\bfxi_\cdot$ given in \eqref{eqn:xi}  is an $\epsilon$-NE to the game \eqref{eq:n-fv} where $\epsilon=O\left(\frac{1}{\sqrt N}\right)$.
\end{thm}
\begin{proof}
It suffice to show that
\[\eee_{\mu_0}\left[J^i(\bfx_{0-},\bar\xi^{i}_\cdot;\bar\bfxi^{-i}_\cdot)\right]\geq\sup_{\xi^i_\cdot\in\mathcal U^N}\eee_{\mu_0}\left[J^i(\bfx_{0-},\xi^{i}_\cdot;\bar\bfxi^{-i}_\cdot)\right]-O\left(\frac{1}{\sqrt N}\right).\]
Note that the strategies of other companies are fixed as $\bar\xi^j_\cdot$, where $j\neq i$. By the continuity of $\tilde\rho(\cdot)$ and boundedness of $\bar x^j$'s, $\bar \rho:=\frac{\sum_{j\neq i}\tilde\rho(\bar x^j_\infty)}{N-1}$ is bounded by a sufficiently large number $R>0$. Therefore for any $\xi^i_\cdot\in\mathcal U^N$,
\[J^i(\bfx,\xi^i_\cdot;\bar\bfxi^{-i}_\cdot)\leq J^{u,i}(\bfx,\xi^i_\cdot;\bar\bfxi^{-i}_\cdot):=\eee\left[\int_0^\infty e^{-rt}\left[cR(x^i_t)^{\alpha}dt-pd\xi^{i,+}_t+p(1-\lambda)d\xi^{i,-}_t\right]\biggl|\bfx_{0-}=\bfx\right].\]
From \cite{Guo2005}, clearly $\sup_{\xi^i\in\mathcal U^N}J^{u,i}(\bfx,\xi^i_\cdot;\bar\bfxi^{-i}_\cdot)$ is finite and
\[U:=\sup_{\xi^i\in\mathcal U^N}J^{u,i}(\bfx,\xi^i_\cdot;\bar\bfxi^{-i}_\cdot)<\infty.\]
For some $d>0$ such that $U-\frac{d}{\sqrt N}>0$, take $\hat\xi^i_\cdot\in\mathcal U^N$ such that $J^i(\bfx,\hat\xi^i_\cdot;\bar\bfxi^{-i}_\cdot)\geq U-\frac{d}{\sqrt N}$ and denote the production level under policy $\hat\xi^i_\cdot$ by $\hat x^i=\{\hat x^i_t\}_{t\geq0}$. According to \eqref{eq:n-adm-ctrl-set}, there exists $L>0$ $\eee_{\mu_0}\left[\int_0^\infty e^{-rt}(\hat x^i_t)^2dt\right]<L$. Consider $\xi^i_\cdot\in\mathcal U^N$ such that the corresponding controlled process $x^i$ satisfies
\begin{equation}\label{eq:l2-bd}
\eee_{\mu_{0}}\left[\int_0^\infty e^{-rt}(x^i_t)^2 dt\right]<L.
\end{equation}
Take such a control policy $\xi^i_\cdot$.
\[\frac{c(x^i_t)^\alpha}{N-1}\sum_{j\neq i}\tilde\rho(\bar x^j_\infty)=c(x^i_t)^\alpha\left\{\rho^*+\frac{\sum_{j\neq i}\left[\tilde\rho(\bar x^j_\infty)-\rho^*\right]}{N-1}\right\}.\]
Since $\bar x^j_\cdot\in[\tilde x^*_b,\tilde x^*_s]$ almost surely, $\tilde\rho(\bar x^i_\infty)$ is also bounded almost surely. For $j\neq i$, $\bar x^j$'s are i.i.d., then
\[\eee_{\mu_{0}}\left|\frac{\sum_{j\neq i}\left[\tilde\rho(\bar x^j_\infty)-\rho^*\right]}{N-1}\right|\leq\eee_{\mu_{0}}\left[\left|\frac{\sum_{j\neq i}\left[\tilde\rho(\bar x^j_\infty)-\rho^*\right]}{N-1}\right|^2\right]^{\frac{1}{2}}=O\left(\frac{1}{\sqrt N}\right).\]
Therefore,
\begin{align*}
&\left|\eee_{\mu_0}\left[\int_0^\infty e^{-rt}\frac{c(x^i_t)^\alpha}{N-1}\sum_{j\neq i}\tilde\rho(\bar x^j_\infty)dt\right]-\eee_{\mu_0}\left[\int_0^\infty e^{-rt}c\rho^*(x^i_t)^\alpha dt\right]\right|\\
&\hspace{20pt}\leq \eee_{\mu_0}\left[\int_0^\infty e^{-rt}c\left|\frac{\sum_{j\neq i}\left[\tilde\rho(\bar x^j_\infty)-\rho^*\right]}{N-1}\right|(x^i_t)^\alpha dt\right]\\
&\hspace{20pt}\leq\eee_{\mu_{0}}\left[\left|\frac{\sum_{j\neq i}\left[\tilde\rho(\bar x^j_\infty)-\rho^*\right]}{N-1}\right|^2\right]^{\frac{1}{2}}\eee_{\mu_0}\left[\left(\int_0^\infty e^{-rt}(x^i_t)^\alpha dt\right)^2\right]^{\frac{1}{2}}\\
&\hspace{20pt}\leq\eee_{\mu_{0}}\left[\left|\frac{\sum_{j\neq i}\left[\tilde\rho(\bar x^j_\infty)-\rho^*\right]}{N-1}\right|^2\right]^{\frac{1}{2}}\eee_{\mu_0}\left[\frac{1}{r}\int_0^\infty e^{-rt}(x^i_t)^{2\alpha} dt\right]^{\frac{1}{2}},
\end{align*}
where the last inequality is by the Jensen inequality. By \eqref{eq:l2-bd}, we have
\begin{align*}
\eee_{\mu_0}\left[\int_0^\infty e^{-rt}(x^i_t)^{2\alpha} dt\right]&=\eee_{\mu_0}\left[\int_0^\infty e^{-rt}(x^i_t)^{2\alpha}\mathbbm 1\{x^i_t\leq1\} dt\right]+\eee_{\mu_0}\left[\int_0^\infty e^{-rt}(x^i_t)^{2\alpha}\mathbbm 1\{x^i_t>1\} dt\right]\\
&\leq r+\eee_{\mu_{0}}\left[\int_0^\infty e^{-rt}(x^i_t)^2 dt\right]\\
&\leq r+L.
\end{align*}
Therefore, 
\[\eee_{\mu_0}\left[\int_0^\infty e^{-rt}\frac{c(x^i_t)^\alpha}{N-1}\sum_{j\neq i}\tilde\rho(\bar x^j_\infty)dt\right]=\eee_{\mu_0}\left[\int_0^\infty e^{-rt}c\rho^*(x^i_t)^\alpha dt\right]+O\left(\frac{1}{\sqrt N}\right).\]
In particular,
\begin{align*}
&\sup_{\xi^i_\cdot\in\mathcal U^N}\eee_{\mu_0}\left[J^i(\bfx_{0-},\xi^{i}_\cdot;\bar\bfxi^{-i}_\cdot)\right]-O\left(\frac{1}{\sqrt N}\right)\leq \eee_{\mu_0}\left[J^i(\bfx_{0-},\hat\xi^{i}_\cdot;\bar\bfxi^{-i}_\cdot)\right]\\
&\hspace{20pt}=\eee_{\mu_0}\left[\int_0^\infty e^{-rt}\left[\frac{c(\hat x^i_t)^{\alpha}}{N-1}\sum_{j\neq i}\tilde\rho(x^{j}_\infty)dt-\gamma^+d\xi^{i,+}_t-\gamma^-d\xi^{i,-}_t\right]\right]\\
&\hspace{20pt}=\eee_{\mu_0}\left[\int_0^\infty e^{-rt}\left[c\rho^*(\hat x^i_t)^{\alpha}dt-\gamma^+d\xi^{i,+}_t-\gamma^-d\xi^{i,-}_t\right]\right]+O\left(\frac{1}{\sqrt N}\right)\\
&\hspace{20pt}\leq \eee_{\mu_0}\left[J^i(\bfx_{0-},\bar\xi^{i}_\cdot;\bar\bfxi^{-i}_\cdot)\right]+O\left(\frac{1}{\sqrt N}\right),
\end{align*}
where the last inequality is due to the optimality of $\bar\xi^i_\cdot$ according to Step 1 of Section \ref{ssec:exp-soln}.
\end{proof}

\section{Conclusion and remarks}

This paper analyzes a class of MFGs with singular controls motivated from the partially reversible problem.
It provides an explicit solution to the MFG, presents sensitivity analysis to compare the solution to the MFG with that of the single-agent control problem, and establishes its approximation to the corresponding $N$-player game in the sense of $\epsilon$-NE, with $\epsilon=O\left(\frac{1}{\sqrt{N}}\right)$.

The natural next step is to study the problem of convergence of the $N$-player game to the associated MFG. Note that this problem  has been studied for  regular controls in \cite{Lacker2018a, Cardaliaguet2018, Nutz2018}. It will be interesting to explore the case when controls are possibly discontinuous. 

Another class of stochastic games with possibly discontinuous controls is impulse control games. Recently there are  progresses in this direction, including \cite{aid2020nonzero} and \cite{campi2020nonzero} for explicit solutions of two-player games and \cite{basei2019nonzero}   showing solutions of  impulse MFGs being $\epsilon$-NE for their corresponding $N$-player impulse games, with $\epsilon=O(\frac{1}{\sqrt{N}})$.  Similar to  MFGs with singular controls, it is challenging to establish general NE structures for impulse games, except for some special cases. The main challenge comes from the non-local operator associated with impulse controls, even with one-dimensional state processes. 

Finally, it is well known that under proper technical conditions, singular controls of  finite variation can be approximated by singular controls of bounded velocity.
See for instance \cite{Hern-Hern2016}. 
More recently, \cite{Dianetti2019} studies in an $N$-player game setting  the convergence of singular control of bounded velocity to that of finite variation, assuming sub-modularity of the cost function and via the notion of {\em weak} NE. An immediate question is whether the convergence relation holds in a MFG framework, and if so, under what form of equilibrium. This is an intriguing question beyond the scope of this paper.

\bibliographystyle{alpha}
\bibliography{library-2}
\end{document}